\documentclass[final]{article}

\usepackage[utf8]{inputenc}
\usepackage{amsmath, amsfonts, amsthm}
\usepackage{graphicx}
\usepackage{hyperref}
\usepackage{booktabs}
\usepackage{enumerate}
\usepackage{color}
\usepackage[top=2cm, bottom=2cm, left=2cm, right=2cm]{geometry}
\usepackage[square,numbers,sort&compress]{natbib}
\usepackage{jabbrv}

\usepackage[normalem]{ulem}

\usepackage[notref, notcite]{showkeys}

\usepackage{algorithm,algpseudocode}

\usepackage{fontawesome}

\newcommand{\dxk}{\Delta{x}^k}
\newcommand{\dyk}{\Delta{\lambda}^k}
\newcommand{\dsk}{\Delta{z}^k}
\newcommand{\dXk}{\Delta{X}^k}
\newcommand{\dSk}{\Delta{Z}^k}

\newcommand{\ssz}{\bar \alpha}
\newcommand{\sszk}{\bar \alpha_k}
\newcommand{\dxkm}[1]{\Delta{x}^{k + #1}}
\newcommand{\dykm}[1]{\Delta{\lambda}^{k + #1}}
\newcommand{\dskm}[1]{\Delta{z}^{k + #1}}
\newcommand{\dxkmm}[1]{\Delta{x}^{k - #1}}
\newcommand{\dykmm}[1]{\Delta{\lambda}^{k - #1}}
\newcommand{\dskmm}[1]{\Delta{z}^{k - #1}}
\newcommand{\dXkm}[1]{\Delta{X}^{k + #1}}
\newcommand{\dSkm}[1]{\Delta{Z}^{k + #1}}
\newcommand{\pxysk}[1]{x^{#1}, {\lambda}^{#1}, z^{#1}}
\newcommand{\xysk}[1]{(\pxysk{#1})}
\newcommand{\pdxysk}{\dxk, \dyk, \dsk}
\newcommand{\pdxyskm}[1]{\dxkm{#1}, \dykm{#1}, \dskm{#1}}
\newcommand{\dxysk}{(\pdxysk)}
\newcommand{\dxyskm}[1]{(\pdxyskm{#1})}
\newcommand{\R}{\mathbb{R}}
\newcommand{\hopdm}{\texttt{HOPDM}}

\newcommand{\nN}{\mathcal{N}}
\newcommand{\F}{\mathcal{F}}

\newcommand{\alphadec}{\alpha_\mathrm{dec}}

\DeclareMathOperator{\diag}{diag}

\newtheorem{assumption}{Assumption}
\newtheorem{lemma}{Lemma}
\newtheorem{theorem}{Theorem}
\newtheorem{corollary}{Corollary}

\title{Polynomial worst-case iteration complexity of quasi-Newton \\
  primal-dual interior point algorithms for linear programming}

\author{%
  \href{https://orcid.org/0000-0002-6270-4666}{\includegraphics[scale=0.6]{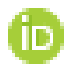}}
  J. Gondzio\thanks{School of Mathematics, University of Edinburgh,
    Edinburgh, EH9 3FD, Scotland, United Kingdom. Email:
    \url{J.Gondzio@ed.ac.uk}} \and%
\href{https://orcid.org/0000-0003-4963-0946}{\includegraphics[scale=0.6]{figure1.eps}}
 F. N. C. Sobral\thanks{{\faEnvelopeO}~Corresponding author. Department of
    Mathematics, State University of Maringá, Avenida Colombo, 5790,
    Paraná, Brazil, 87020-900. Phone: +55 44 30116211. E-mail:
    \url{fncsobral@uem.br}}%
}

\date{Technical Report ERGO, School of Mathematics, September 9, 2022}

\begin{document}

\maketitle

\begin{abstract}
  Quasi-Newton methods are well known techniques for large-scale
  numerical optimization. They use an approximation of the Hessian 
  in optimization problems or the Jacobian in system of nonlinear
  equations. In the Interior Point context, quasi-Newton algorithms
  compute low-rank updates of the matrix associated with the Newton
  systems, instead of computing it from scratch at every iteration. 
  In this work, we show that a simplified quasi-Newton primal-dual
  interior point algorithm for linear programming enjoys polynomial 
  worst-case iteration complexity. 
  Feasible and infeasible cases of the algorithm are considered and 
  the most common neighborhoods of the central path are analyzed. 
  To the best of our knowledge, this is the first attempt to deliver 
  polynomial worst-case iteration complexity bounds for these methods.  
  Unsurprisingly, the worst-case complexity results obtained 
  when quasi-Newton directions are used are worse than their 
  counterparts when Newton directions are employed.  
  However, quasi-Newton updates are very attractive for large-scale 
  optimization problems where the cost of factorizing the matrices 
  is much higher than the cost of solving linear systems. 
\end{abstract}

\section{Introduction} \label{intro}

Let us consider the following general linear programming problem
\begin{equation}
  \label{def:qp}
  \min \quad c^T x, \quad {\mbox {s.t.} } \quad A x = b, \ x \geq 0,
\end{equation}
where $x, c \in \R^n$, $b \in \R^m$ and $A \in \R^{m \times n}$. 
We assume that~\eqref{def:qp} is feasible and the rows of $A$ are
linearly independent.
Define function $F: \R^{2n + m} \to \R^{2n + m}$ by
\begin{equation}
  \label{def:f}
  F(x, \lambda, z) =
  \begin{bmatrix}
    A^T \lambda + z - c \\
    A x - b \\
    X Z e
  \end{bmatrix},
\end{equation}
where $X, Z \in \R^{n \times n}$ are diagonal matrices defined by
$X = \diag(x)$ and $Z = \diag(z)$, respectively, and $e$ is the vector
of ones of appropriate size. First order necessary optimality conditions
for~\eqref{def:qp} state that, if $x^* \ge 0$ is a minimizer, then
there exist $z^* \in \R^n$, $z^* \ge 0$, and $\lambda^* \in \R^m$ such
that $F(\pxysk{*}) = 0$ holds.

Interior point methods (IPMs) try to follow the so-called central-path
of problem~\eqref{def:qp}, defined by the solution of the perturbed 
KKT system
$F(\pxysk{}) = \begin{bmatrix} 0 & 0 & \mu e \end{bmatrix}^T$, as
$\mu \to 0$. Instead of solving such a system exactly, primal-dual IPMs 
apply one iteration of Newton method for a given value of $\mu_k$ at
iteration $k$. In order to calculate this step, the Jacobian of $F$ is
needed
\begin{equation}
  \label{def:j}
  J(x, \lambda, z) = 
  \begin{bmatrix}
    0 & A^T & I \\
    A &   0 & 0 \\
    Z &   0 & X    
  \end{bmatrix} .
\end{equation}
With an iterate $\xysk{k}$ at step $k$, the
classical Newton direction is calculated by solving the following system
\begin{equation}
  \label{def:newtonsys}
    \begin{bmatrix}
    0   & A^T & I \\
    A   & 0   & 0 \\
    Z^k & 0   & X^k \\
  \end{bmatrix}
  \begin{bmatrix}
    \dxk \\ \dyk \\ \dsk
  \end{bmatrix}
  =
  \begin{bmatrix}
    c - z^k - A^T \lambda^k \\ b - A x^k \\  \sigma_k \mu_k e - X^k Z^k e
  \end{bmatrix}
  ,
\end{equation}
where $\mu_k$ is set to be the average complementarity gap ${x^k}^T s^k / n$ 
and $\sigma_k \in (0,1)$ determines its target reduction.

While the coefficient matrix in~\eqref{def:newtonsys} can be
efficiently evaluated and stored, it changes at each iteration.  The
solution of~\eqref{def:newtonsys} is usually accomplished by direct
methods, using suitable matrix
factorizations~\cite{GondzioGrothey2009,LustigEtAl1994}, or by
iterative methods~\cite{DApuzzoElAl2010}, computing preconditioners to
improve their convergence properties. In this paper we are concerned 
with classes of problems for which it is advantageous to approximate
$J(\pxysk{k})$ in order to reduce the cost of
solving~\eqref{def:newtonsys}.

Usually, IPMs do not deal explicitly with the unreduced 
system~\eqref{def:newtonsys}, but rather consider its reduced
form as augmented system (which is symmetric) or as normal equations
(whose coefficient matrix is positive definite)~\cite{Gondzio2012a}. The interest in
working directly with unreduced systems has attracted more attention
in the recent years, since they have good sparsity structure and also
interesting spectral
properties~\cite{Greif2014}. In~\cite{Morini2017}, numerical
experiments comparing preconditioners for unreduced and augmented
systems were made. The appeal for using preconditioners for unreduced
systems is their good conditioning close to the solution.

Gondzio and Sobral~\cite{GondzioSobral2019} considered unreduced
systems in a way similar to~\cite{Morini2017}. They studied 
the Jacobian of $F$ and asked the question whether it is possible 
to approximate it by classical quasi-Newton approaches for 
nonlinear systems. Although this might have seemed an obvious thing 
to attempt, the only previous use of quasi-Newton strategies in IPMs 
was to update the 
preconditioners~\cite{Gratton2011,Bellavia2015,Gratton2016,Bergamaschi2018}. 
Broyden low-rank updates were used and the numerical experiments showed 
that this approach is effective for IPMs when the cost of solving linear
systems is considerably lower than the cost of computing the factorization 
of the Jacobian (or its associated reduced form).

Recently, Ek and Forsgren~\cite{EkForsgren2021b} presented a
theoretical background and numerical experiments regarding a different
kind of low rank updates. The proposed update is based on the
Eckart-Young-Mirsky theorem, rather than on the secant equation
satisfied by the Broyden update, and affects only the ``third row'' 
of matrix $J(\pxysk{k})$, related to the nonlinear part of $F$. Convex
quadratic optimization problems were considered and local convergence 
was established for a simplified primal-dual interior point algorithm, 
but no complexity bound was provided. 
It is worth mentioning that the iteration worst-case complexity 
of $O(\sqrt{n})$ was shown for a short-step primal algorithm 
by Gonzaga~\cite{Gonzaga1989}, where low-rank updates were used 
to compute the projection matrix needed by such type of algorithms. 
Secant equations were also used in~\cite{Dennis1987} for the same 
purpose, but without complexity results.

Polynomial worst-case iteration complexity is a key feature of IPMs
for linear and convex-quadratic problems~\cite{NesNem1994}. It is achieved 
by taking steps in the Newton direction~\eqref{def:newtonsys} such that 
the new iterate belongs to some neighborhood of the central path. 
In case of linear programming, it is well known that the iteration 
worst-case complexity involves polynomials of orders between $\sqrt{n}$ 
and $n^2$, depending on the type of neighborhood of the central path used 
and whether feasible of infeasible iterates are allowed~\cite{Wright1997}. 
Those results have also been generalized to symmetric cone optimization
problems~\cite{Schmieta2003}. 

This work is intended to provide the first steps towards the study of
iteration worst-case complexity of quasi-Newton primal-dual interior
point algorithms. We present non-trivial extensions of well known
complexity results from~\cite{Wright1997} and properties that arise
when Broyden ``bad'' quasi-Newton updates are used. Worst-case
complexity is proven for both feasible and infeasible cases in the
most commonly used neighborhoods. The theoretical study is motivated
by the very promising results from~\cite{GondzioSobral2019} in
quadratic programming problems. As expected, the degrees of polynomials 
in the complexity results are higher than those obtained when steps 
in Newton directions are made.

The paper is organized as follows. In Section~\ref{back} we review 
basic quasi-Newton concepts and the properties of quasi-Newton 
algorithms presented in~\cite{GondzioSobral2019}. 
Then, in Section~\ref{feasible} we analyze the worst-case complexity 
for the feasible case, considering two popular 
neighborhoods of the central path : $\nN_2$      
and $\nN_s$. 
Section~\ref{infeas} is devoted to the infeasible case when 
the iterates are confined to the $\nN_s$ neighborhood. 
Final comments, observations and possible directions of future 
work are discussed in Section~\ref{remarks}.

\paragraph{Notation}
We define $\| \cdot \|$ as the Euclidean norm for vectors and the
induced $\ell_2$-norm for matrices. We will use the short versions
$F_k$ and $J_k$ to describe $F(\pxysk{k})$ and $J(\pxysk{k})$,
respectively. In addition, we will use both inline
  $\xysk{k}$ and matrix
  $\left[\begin{smallmatrix} x^k \\ \lambda^k \\ z^k \end{smallmatrix} \right]$ 
  notations to address vectors in this work.
%

%
\section{Background}
\label{back}

Given a function $G:\R^N \to \R^N$, suppose that we want to solve the
nonlinear system $G(\bar x) = 0$. Secant methods iteratively construct
a linear model $M_k(\bar x)$ of $G$ which interpolates the last two
computed iterates of the method. At each iteration, they need to compute 
an approximation to the Jacobian of $G$, which has to satisfy the secant 
equation
\[
  B s_{k-1} = y_{k - 1},
\]
where $s_{k - 1} = \bar x^k - \bar x^{k - 1}$ and
$y_{k - 1} = G(\bar x^{k}) - G(\bar x^{k - 1})$. There are infinitely
many solutions to the secant equation for $N \ge 2$ and different
approaches generate different secant
methods~\cite{Martinez2000b}. Among them, the Broyden ``bad'' approach
uses the already computed approximation to the inverse of $G$ at
$\bar x^{k - 1}$, called $H_{k - 1}$, to compute the current
approximation $H_k$ as
\begin{equation}
  \label{bbroyd1}
  H_k = H_{k - 1} + \frac{(s_{k - 1} - H_{k - 1} y_{k - 1}) y_{k - 1}^T}{y_{k - 1}^T y_{k - 1}} =
  H_{k - 1} V_{k - 1} + \frac{s_{k - 1} y_{k - 1}^T}{\rho_{k - 1}},
\end{equation}
where
$V_{k - 1} = \left( I - \frac{y_{k - 1} y_{k - 1}^T}{\rho_{k - 1}}
\right)$ and $\rho_{k - 1} = y_{k - 1}^T y_{k - 1}$.
The Broyden ``bad'' update is a rank-1 update where $H_k$ is the
matrix closest to $H_{k - 1}$ in the Frobenius norm which satisfies
the secant equation. After $\ell$ updates of an approximation
$H_{k - \ell}$, current approximation $H_{k}$ is given by
\begin{equation}
  \label{bbroyd2}
  \begin{split}
    H_{k}
    &  = H_{k - 1} V_{k - 1} + \frac{s_{k - 1} y_{k -
        1}^T}{\rho_{k - 1}} \\
    & = H_{k - \ell} \left( \prod_{j = k - \ell}^{k - 1}
      V_{j} \right) + \sum_{i = 1}^{\ell} \left( \frac{s_{ k - i}
        y_{ k - i}^T}{\rho_{ k - i}}
      \prod_{j = k - i + 1}^{k - 1} V_{j} \right).
  \end{split}
\end{equation}

For the specific case of this work, where $G$ is given by $F$ defined
in~\eqref{def:f}, we have that $N = 2n + m$, $\bar x = (\pxysk{})$ and
the vectors $s_{k - 1}$ and $y_{k - 1}$ from the secant equation
assume a more specific description
\begin{equation}
  \label{def:skyk}
    s_{k - 1}
    = \ssz_{k - 1}
    \begin{bmatrix}
      \dxkmm1 \\ \dykmm1 \\ \dskmm1
    \end{bmatrix}\quad \text{and}\quad
    y_{k - 1}
    =
    \begin{bmatrix}
      \ssz_{k - 1} (A^T \dykmm1 + \dskmm1) \\
      \ssz_{k - 1} A \dxkmm1 \\
      X^{k} Z^{k} e - X^{k - 1} Z^{k - 1} e
    \end{bmatrix},
\end{equation}
where $\bar \alpha_{k - 1} \in (0, 1]$ is the step-size taken at iteration
$k - 1$ towards the solution of~\eqref{def:newtonsys}.

In~\cite{GondzioSobral2019}, the authors described an interior point 
method based on low rank quasi-Newton approximations to the Jacobian 
of $F$. The Broyden updates were tested, and the computational 
experience revealed that the most efficient one was the Broyden 
``bad'' update. 

Since we are interested in finding an approximate solution of the linear 
system given by the Newton method (\ref{def:newtonsys}), in the Broyden 
``bad'' approach, given $\ell \ge 0$, the following direction is computed 
\begin{equation}
  \label{def:qnsys}
  \begin{bmatrix}
    \dxk \\ \dyk \\ \dsk
  \end{bmatrix}
  =
  H_k
  \begin{bmatrix}
    c - A^T \lambda^k - z^k \\ 
    b - A x^k \\ 
    \sigma_k \mu_k e - X^k Z^k e
  \end{bmatrix}.
\end{equation}
If $H_{k - \ell} = J_{k - \ell}^{-1}$ and $\ell = 0$,
system~\eqref{def:qnsys} turns out to be
exactly~\eqref{def:newtonsys}. Therefore, in the same way as discussed
in~\cite{GondzioSobral2019}, we assume that the initial
approximation $H_{k - \ell}$ is given by the perfect approximation
$J_{k - \ell}^{-1}$. When $\ell > 0$, the quasi-Newton procedure
strongly uses the fact that the factorization of $J_{k - \ell}$ (or a
good preconditioner) has already been computed. In addition,
Lemma~\ref{l:bbroyd}, taken from~\cite{GondzioSobral2019}, shows that,
with this choice of initial approximation, the Broyden ``bad'' update
has an interesting alternative interpretation.

\begin{lemma}
  \label{l:bbroyd}
  Assume that $k, \ell \ge 0$ and $H_{k}$ is the approximation of
  $J_{k}^{-1}$ constructed by $\ell$ updates~\eqref{bbroyd2} using
  initial approximation $H_{k - \ell} = J_{k - \ell}^{-1}$.  Given
  $v \in \R^{2n + m}$, the computation of $r = H_k v$ is equivalent to
  the solution of
  \begin{equation*}
    \label{l:bbroyd:e1}
    J_{k - \ell} r = v + 
    \begin{bmatrix}
      0 \\ 0 \\ \sum \limits_{i = 1}^{\ell} \gamma_i \left[ \bar
        \alpha_{k - i} \left(Z^{k - \ell} \dxkmm{i} + X^{k - \ell}
          \dskmm{i} \right) - \left(X^{k - i + 1} Z^{k - i + 1} - X^{k
            - i} Z^{k - i}\right) e \right]
    \end{bmatrix},
  \end{equation*}
  where
  $\displaystyle \gamma_i = \frac{y_{k - i}^T \prod_{j = k - i + 1}^{k
      - 1} V_{j}}{\rho_{k - i}} v$, for $i = 1, \dots, \ell$.
\end{lemma}

Lemma~\ref{l:bbroyd} is the basis of the analysis developed in this
work. It states that we can study quasi-Newton steps using the
Jacobian of the Newton step. The only difference is the right-hand
side. Using this property of the Broyden ``bad'' update we are 
able to extend the well known complexity results described
in~\cite{Wright1997}. The difficulty in the analysis will be mostly 
caused by the extra term, added to the usual right-hand side 
of~\eqref{def:newtonsys}. It is important to note that
Lemma~\ref{l:bbroyd} does not assume that the iterates are feasible,
hence it is useful in both feasible and infeasible cases. Although
by~\eqref{bbroyd2}, the sparsity structure of the third row of $B_k$ 
(the inverse of $H_k$) is lost when $\ell \ge 1$, we can see that 
the structural sparsity of $J_{k - \ell}$ can still be used to solve 
the linear systems.

Let us define a skeleton primal-dual quasi-Newton interior point algorithm. 
It is given by Algorithm~\ref{alg:pdqnipm} and generates a sequence 
of alternating Newton and quasi-Newton steps. Clearly, by the nature 
of update~\eqref{bbroyd1}, the first step needs to be a Newton step.

\begin{algorithm}

  \begin{algorithmic}[0]

    \State \textbf{Input:} $F$, $J$ and $\xysk{0}$
    \For{$k = 0, 1, \dots$}
    \If{$k$ is odd}
      \State $\ell \gets 1$
      \Comment{Quasi-Newton iteration}
    \Else
      \State $\ell \gets 0$
      \Comment{Newton iteration}
    \EndIf

    \State Calculate $\dxysk$ by
    solving~\eqref{def:qnsys}\;

    \State Calculate
    \[
      \xysk{k + 1} = \xysk{k} + \bar \alpha_k \dxysk
    \]
    
    \State for a suitable choice of $\bar \alpha_k \in [0, 1]$, such that
    $x^{k + 1}, z^{k + 1} > 0$

    \EndFor
  \end{algorithmic}







    


  




  \caption{Conceptual Quasi-Newton Interior Point algorithm.}
  \label{alg:pdqnipm}
\end{algorithm}
 
Let us analyze what happens when a sequence of two steps 
is performed: at iteration $k$ the Newton step is made 
(with stepsize $\bar \alpha_k$)
and then at iteration $k+1$ the quasi-Newton step is taken 
(with stepsize $\bar \alpha$). 
For Newton step at iteration $k$, we observe that
\begin{equation}
  \label{CProds-kp1}
  \begin{split}
    X^{k + 1} Z^{k + 1} e & = \left(X^k + \bar \alpha_k \dXk\right)
    \left(Z^k + \bar \alpha_k \dSk\right) e \\
    & = X^k Z^k e + \bar \alpha_k \left( Z^k \dxk + X^k \dsk \right) 
        + {\bar \alpha_k}^2 \dXk \dSk e \\
    & = X^k Z^k e + \bar \alpha_k ( \sigma_k \mu_k e - X^k Z^k e )
        + {\bar \alpha_k}^2 \dXk \dSk e \\
    & = (1 - \bar \alpha_k) X^k Z^k e + \bar \alpha_k \sigma_k \mu_k e
        + {\bar \alpha_k}^2 \dXk \dSk e.
  \end{split}
\end{equation}
Later, in the proofs of several technical results, we will need 
to analyze the error produced when the quasi-Newton direction
$\dxyskm1$ 
is multiplied by $J_{k + 1}$:  
  \[
    \begin{bmatrix}
      0 & A^T & I \\
      A & 0   & 0 \\
      Z^{k + 1} & 0   & X^{k + 1} \\
    \end{bmatrix}
    \begin{bmatrix}
      \dxkm1 \\ \dykm1 \\ \dskm1
    \end{bmatrix}
    =
    \left(
      \begin{bmatrix}
        0 & 0 & 0 \\
        0 & 0   & 0 \\
        Z^{k + 1} - Z^k & 0   & X^{k + 1} - X^k \\
      \end{bmatrix}
      +
      \begin{bmatrix}
        0 & A^T & I \\
        A & 0   & 0 \\
        Z^{k} & 0   & X^{k} \\
      \end{bmatrix}
    \right)      
    \begin{bmatrix}
      \dxkm1 \\ \dykm1 \\ \dskm1
    \end{bmatrix}.
  \]
  Applying Lemma~\ref{l:bbroyd} for iteration $k + 1$ with $\ell = 1$
  and then observing that $\dxysk$ solves the Newton
  system~\eqref{def:newtonsys} and using~\eqref{CProds-kp1}, the third block
  equation in Lemma~\ref{l:bbroyd} gives
  \begin{equation}
    \label{useful-result} 
    \begin{split}
      Z^k \dxkm1 + X^k \dskm1
    & = \sigma_{k + 1} \mu_{k + 1} e - X^{k + 1} Z^{k + 1} e 
        + \gamma_1 \left[ \sszk \left(Z^{k} \dxk + X^{k} \dsk \right) 
        - \left(X^{k + 1} Z^{k + 1} - X^{k} Z^{k}\right) e \right] \\
    & = \sigma_{k + 1} \mu_{k + 1} e - X^{k + 1} Z^{k + 1} e 
        + \gamma_1 (\bar \alpha_k \sigma_k \mu_k e
        + (1 - \bar \alpha_k) X^k Z^k e - X^{k + 1} Z^{k + 1} e) \\
    & = \sigma_{k + 1} \mu_{k + 1} e - X^{k + 1} Z^{k + 1} e 
        - \gamma_1 {\bar \alpha_k}^2 \dXk \dSk e.
    \end{split}
  \end{equation}
  Hence, using~\eqref{def:skyk} and~\eqref{useful-result}
  \begin{equation}
    \label{errorZdXpXdZ}
    \begin{split}
      Z^{k + 1} \dxkm1 & + X^{k + 1} \dskm1 = (Z^{k + 1} - Z^k) \dxkm1
      + (X^{k + 1} - X^k) \dskm1 + Z^k \dxkm1 + X^k \dskm1 \\
      & = \sszk \left(\dSk \dXkm1 e + \dXk \dSkm1 e\right) + \sigma_{k + 1}
      \mu_{k + 1} e - X^{k + 1} Z^{k + 1} e - \gamma_1 \sszk^2 \dXk \dSk e,
    \end{split}
  \end{equation}
  where $\dSk$, $\dXk$, $\dSkm1$ and $\dXkm1$ are given by
  $\diag(\dsk)$, $\diag(\dxk)$, $\diag(\dskm1)$ and $\diag(\dxkm1)$,
  respectively.
Next we compute the new complementarity products obtained after 
a sequence of two steps, apply (\ref{errorZdXpXdZ}), and add and 
subtract the term $\bar \alpha_k^2 \dXk \dSk e$ to derive 
\begin{equation}
  \label{NewComplProds}
    \begin{split}
      X^{k + 2} Z^{k + 2} e & = 
        (X^{k+1} + \bar \alpha \dXkm1) (Z^{k+1} + \bar \alpha \dSkm1) e \\
      & = X^{k+1} Z^{k+1} e 
        + \bar \alpha (Z^{k+1} \dXkm1 + X^{k + 1} \dSkm1) e 
        + \bar \alpha^2 \dXkm1 \dSkm1 e \\
      & = X^{k+1} Z^{k+1} e 
        + \bar \alpha \bar \alpha_k ( \dSk \dXkm1 e + \dXk \dSkm1 e ) 
        + \bar \alpha^2 \dXkm1 \dSkm1 e 
        + \bar \alpha_k^2 \dXk \dSk e \\
      & \quad - \bar \alpha_k^2 \dXk \dSk e 
        + \bar \alpha \sigma_{k + 1} \mu_{k + 1} e - \bar \alpha X^{k + 1} Z^{k + 1} e 
        - \gamma_1 \bar \alpha {\bar \alpha_k}^2 \dXk \dSk e \\
      & = (1 - \bar \alpha) X^{k+1} Z^{k+1} e 
        + (\bar \alpha_k \dXk + \bar \alpha \dXkm1) (\bar \alpha_k \dSk + \bar \alpha \dSkm1) e \\
      & \quad + \bar \alpha \sigma_{k + 1} \mu_{k + 1} e 
        - (1 + \bar \alpha \gamma_1) {\bar \alpha_k}^2 \dXk \dSk e. 
    \end{split} 
\end{equation}
By multiplying both sides of equation~\eqref{NewComplProds} 
with $e^T$ we get the complementarity product at iteration $k+2$: 
\begin{equation} 
  \label{CP-afterQN}
  \begin{split}
    (x^{k + 1} & + \ssz \dxkm1)^T (z^{k + 1} + \ssz \dskm1) = \\
    & = (1 - \ssz (1 - \sigma_{k + 1})) {x^{k + 1}}^T z^{k
      + 1} + (\sszk \dxk + \ssz \dxkm1)^T (\sszk \dsk + \ssz \dskm1)
    - (1 + \gamma_1 \ssz) \sszk^2 {\dxk}^T \dsk. 
  \end{split}
\end{equation}
It is worth noting that the final expression in (\ref{NewComplProds})
involves a composite direction
$(\bar \alpha_k \dxk + \bar \alpha \dxkm1, \bar \alpha_k \dsk + \bar
\alpha \dskm1)$ which corresponds to an aggregate of two consecutive
steps: in Newton direction at iteration $k$ and in quasi-Newton
direction at iteration $k+1$. Much of the effort of the analysis
presented in this paper is focused on this composite
direction. 
Let us mention that we will 
also use the component-wise versions of equations (\ref{CProds-kp1}), 
(\ref{errorZdXpXdZ}) and (\ref{NewComplProds}). 
For example, in case of (\ref{NewComplProds}) this gives 
\begin{equation}
  \label{singleProd}
  \begin{split}
    \left[ x^{k + 1} \right. & + \left. \bar \alpha \dxkm1 \right]_i
    \left[ z^{k + 1} + \bar \alpha \dskm1 \right]_i 
      = (1 - \bar \alpha) \left( x^{k + 1}_i z^{k + 1}_i \right) \\
    & + \left[ \bar \alpha_k \dxk + \bar \alpha \dxkm1 \right]_i
        \left[ \bar \alpha_k \dsk + \bar \alpha \dskm1 \right]_i 
      + \bar \alpha \sigma_{k + 1} \mu_{k + 1} 
      - (1 + \bar \alpha \gamma_1) {\bar \alpha_k}^2 [\dxk]_i [\dsk]_i.
  \end{split}
\end{equation}

Observe that equations (\ref{CProds-kp1})-(\ref{singleProd}) 
are valid for both feasible and infeasible algorithm. However, the analysis 
in Sections~\ref{feasible} and~\ref{infeas} will distinguish between 
these two cases because in the feasible one we are able to take advantage 
of the orthogonality of primal and dual directions and exploit it 
to deliver better final worst-case complexity results. 

Before we conclude the brief background section and take the reader 
through a detailed analysis of different versions of the primal-dual 
quasi-Newton interior point algorithm, let us observe that 
equation~\eqref{NewComplProds} involves an important term $\gamma_1$.
By Lemma~\ref{l:bbroyd}, $\gamma_1$ can be seen as the scalar 
coefficient of the projection of vector $v$ onto the subspace 
generated by vector $y_k$:
\[
  \mathcal{P}_{y_k}(v) = \frac{y_k^T v}{y_k^T y_k} y_k = \gamma_1 y_k. 
\]
Using the non-expansive property of projections we conclude that
\begin{equation}
  \label{projection}
  \| \mathcal{P}_{y_k}(v) \| \le \|v\| 
  \iff \| \gamma_1 y_k \| \le \| v \| 
  \iff |\gamma_1| \le \frac{\|v\|}{\|y_k\|}.
\end{equation}
In the next lemma, a lower bound for $\|y_k\|$ is derived. It states
that the denominator of~\eqref{projection} can be bounded away from
zero if a sufficient decrease of $\mu = x^T z / n$ is ensured and
non-null step-sizes are taken. The bound for $\|v\|$ involves the
right-hand side in~\eqref{def:qnsys} and therefore depends on the
feasibility of iterates and on the choice of the centering parameter
$\sigma$.

\begin{lemma}
  \label{l:ykbnd}
  Let $k + 1$ be a quasi-Newton iteration of
  Algorithm~\ref{alg:pdqnipm} and $y_k$ be the quasi-Newton vector
  defined by~\eqref{def:skyk} to construct $H_{k + 1}$ by the Broyden
  ``bad'' update~\eqref{bbroyd2}. Suppose that
  $\mu_{k + 1} \le (1 - \rho_k \sszk) \mu_k$
  holds, for $\sszk, \rho_k \in [0, 1]$. Then
  \[
    \|y_k\| \ge \frac{\rho_k}{2} \sszk \mu_k.
  \]
\end{lemma}

\begin{proof}
  If $\sszk = 0$ or $\rho_k = 0$ the result trivially holds, so we can
  assume that $\sszk, \rho_k \in (0, 1]$. Suppose, by contradiction,
  that $\| y_k \| < \rho_k \sszk \mu_k / 2$. Therefore, by definition
  of $y_k$,
  \begin{equation*}
    \begin{split}
      \| X^{k + 1} Z^{k + 1} e - X^k Z^k e \| \le \| y_k \| 
        < \frac{\rho_k \sszk}{2} \mu_k 
      & \Rightarrow |x^{k + 1}_i z^{k + 1}_i - x^k_i z^k_i| 
        < \frac{\rho_k \sszk}{2} \mu_k, i = 1, \dots, n \\
      & \Rightarrow x^{k + 1}_i z^{k + 1}_i - x^k_i z^k_i 
        > - \frac{\rho_k \sszk}{2} \mu_k, i = 1, \dots, n \\
      & \Rightarrow \mu_{k + 1} - \mu_k 
        > - \frac{\rho_k \sszk}{2} \mu_k,
    \end{split}
  \end{equation*}
  where the last result was obtained by adding up all the $n$ previous
  inequalities and dividing by $n$. By hypothesis we have that
  $\mu_{k + 1} \le (1 - \rho_k \sszk) \mu_k$ and, therefore,
  \[
    - \frac{\rho_k \sszk}{2} \mu_k 
    < \mu_{k + 1} - \mu_k \le - \rho_k \sszk \mu_k, 
  \]
  which implies $\rho_k/2 > \rho_k$ and is a clear absurd.
  Thus, we conclude that
  $\|y_k\| \ge \rho_k \sszk \mu_k / 2$.
\end{proof}

\section{Worst-case complexity in the feasible case}
\label{feasible}

For all the results in this section, we suppose that $\xysk{0}$ is
primal and dual feasible, given by Assumption~\ref{a:pdfeas}.

\begin{assumption}
  \label{a:pdfeas}
  $\xysk{0} \in \mathcal{F} \doteq \{ (x, \lambda, z)\ |\ A x = b, A^T
  \lambda + z = c, x > 0, z > 0\}$.
\end{assumption}

Our analysis follows closely the theory in~\cite{Wright1997}.
Under Assumption~\ref{a:pdfeas} the primal and dual directions 
are orthogonal to each other~\cite[Lemma 5.1]{Wright1997}. 
We show that the same holds for quasi-Newton directions. 

\begin{lemma}
  \label{l:ortho}
  If $k+1$ is a quasi-Newton iteration of Algorithm~\ref{alg:pdqnipm},
  then
  $
  {\dxkm1}^T \dskm1 = 0.
  $
\end{lemma}

\begin{proof}
  Using Lemma~\ref{l:bbroyd} with $r=Hv$ defined by their
  respective terms in~\eqref{def:qnsys} and by the primal and dual
  feasibility of $\xysk0$, we observe that the first two block rows 
  of system~\eqref{def:qnsys} (at iteration $k+1$) given by
  \[
  \left\{
  \begin{array}{rrcrcr}
    & A^T \dykm1 & + & \dskm1 & = & 0 \\
    A \dxkm1 & & & & = & 0 
  \end{array}
  \right.
  \]
  are the same as in the system solved by the usual Newton step. 
  Therefore,
  \[
  {\dxkm1}^T \dskm1 = 
     - {\dxkm1}^T A^T \dykm1 = - \left( A \dxkm1 \right)^T \dykm1 = 0.
  \]
\end{proof}

\begin{corollary}
  \label{c:ortho}
  If $k$ is a Newton iteration and $k+1$ is a quasi-Newton iteration, then
  \[
  \left( \bar \alpha_k \dxk + \bar \alpha \dxkm1 \right)^T 
  \left( \bar \alpha_k \dsk + \bar \alpha \dskm1 \right) =
  {\dxk}^T \dsk = {\dxk}^T \dskm1 = {\dsk}^T \dxkm1 = 0.
  \]
\end{corollary}

\begin{proof}
  Since, by equation~\eqref{def:skyk},
  $\dxysk$ was computed by
  the Newton step at iteration $k$ we have that ${\dxk}^T \dsk = 0$ by
  the same arguments of Lemma~\ref{l:ortho}. By Lemma~\ref{l:bbroyd}
  the first two equations do not change between iterations $k$ and
  $k + 1$, and we have the desired results.
\end{proof}

The second important result is that the quasi-Newton step can decrease
the barrier parameter $\mu$ in exactly the same way as the Newton step
(see~\cite[Lemma~5.1]{Wright1997}). Recall that by definition $\mu = (x^T z) / n$. 

\begin{lemma}
  \label{l:decmu3}
  Let $k + 1$ be a quasi-Newton iteration of Algorithm~\ref{alg:pdqnipm}. 
  Then for any feasible step-size $\bar \alpha \in [0, 1]$
  \begin{equation}
    \label{t:decmu3}
    \mu(\bar\alpha) = (1 - \bar\alpha (1 - \sigma_{k + 1})) \mu_{k + 1}. 
  \end{equation}
\end{lemma}

\begin{proof}
By equation~(\ref{CP-afterQN}) and using Corollary~\ref{c:ortho}, we obtain 
  \begin{equation*}
    \begin{split}
      n \mu (\bar\alpha) & = (x^{k + 2})^T z^{k + 2} = 
        (x^{k + 1} + \bar\alpha \dxkm1)^T (z^{k + 1} + \bar\alpha \dskm1) \\
      & = (1 - \bar \alpha) {x^{k + 1}}^T z^{k + 1} 
        + (\bar \alpha_k \dxk \! + \! \bar \alpha \dxkm1)^T 
          (\bar \alpha_k \dsk \! + \! \bar \alpha \dskm1) 
        + \bar\alpha \sigma_{k + 1} n \mu_{k + 1} 
        - (1 \! + \! \bar\alpha \gamma_1) {\bar \alpha_k}^2 {\dxk}^T \dsk \\
      & = (1 - \bar\alpha (1 - \sigma_{k + 1})) n \mu_{k + 1}. 
    \end{split}
  \end{equation*}
  By dividing both sides of the last equation by $n$ we get the desired result. 
\end{proof}

Then, using Lemma~\ref{l:decmu3}, after Newton step at iteration $k$ 
and quasi-Newton step at iteration $k + 1$, we get 
\begin{equation}
  \label{l:qnstepbnd:e4}
  \begin{split}
    \| (\mu(\ssz) - \mu_k) e \|^2 
    & = (\mu(\ssz) - \mu_k)^2 n 
      \ = \ [(1 - \ssz (1 - \sigma_{k + 1})) \mu_{k + 1} - \mu_k]^2 n \\
    & = [(1 - \ssz (1 - \sigma_{k + 1})) (1 - \sszk (1 - \sigma_k)) \mu_k - \mu_k]^2 n \\
    & = [1 - (1 - \ssz (1 - \sigma_{k + 1})) (1 - \sszk (1 - \sigma_k))]^2 \mu_k^2 n. 
  \end{split}
\end{equation}
%
%
%
%
It is worth noting that (in the feasible case) the term $\gamma_1$ 
originating from Lemma~\ref{l:bbroyd} 
does not have any influence on the value of $\mu(\bar \alpha)$. 

\subsection{The $\nN_2$ neighborhood}
\label{n2}

In this section we will consider a short-step interior point method 
and employ the notion of $\nN_2$ neighborhood of the central path
\[
  \nN_2(\theta) = \{(x, \lambda, z) \in \mathcal{F} \ | 
                    \ \|XZe - \mu e\| \le \theta \mu\},
\]
where $\mathcal{F}$ is the set of primal and dual feasible points 
such that $x, z > 0$, see Assumption~\ref{a:pdfeas}. 
For all considerations in this subsection we add the following assumption.

\begin{assumption}
  \label{a:n2}
  $\xysk{k} \in \nN_2(\theta_k)$ \ and \ 
  $\xysk{k + 1} \in \nN_2(\theta_{k + 1})$, \ for \ 
  $\theta_k, \theta_{k + 1} \in (0, 1)$.
\end{assumption}

Our main goal is to show that the new iterate
\[
  \xysk{k + 2} = \xysk{k + 1} + \bar \alpha \dxyskm1
\]
also belongs to $\nN_2(\theta_{k + 2})$, for suitable choices of
$\theta_{k + 2} \in (0, 1)$, $\bar \alpha$ and $\bar \alpha_k$.
Therefore, we are interested in the analysis of the Euclidean 
norm of the vector
$\, (X^{k + 1} + \bar \alpha \dXkm1) (Z^{k + 1} + \bar \alpha \dSkm1)e - \mu(\bar \alpha) e \,$ 
and to deliver it we will exploit several useful results stated earlier 
in equations (\ref{NewComplProds}), (\ref{singleProd}) 
and Lemma \ref{l:decmu3}. 
Combining (\ref{singleProd}) and Lemma \ref{l:decmu3} we get 
\begin{equation}
  \label{n2:e3}
  \begin{split}
    \left[ x^{k + 1} \right. & + \left. \bar \alpha \dxkm1 \right]_i
    \left[ z^{k + 1} + \bar \alpha \dskm1 \right]_i 
      - \mu(\bar\alpha) = \\
    & = (1 - \bar \alpha) \left( x^{k + 1}_i z^{k + 1}_i - \mu_{k + 1} \right) 
        + \left[ \bar \alpha_k \dxk + \bar \alpha \dxkm1 \right]_i
          \left[ \bar \alpha_k \dsk + \bar \alpha \dskm1 \right]_i 
        - (1 + \bar \alpha \gamma_1) {\bar \alpha_k}^2 [\dxk]_i [\dsk]_i.
  \end{split}
\end{equation}
By~\eqref{n2:e3} and Assumption~\ref{a:n2}, we deliver the following 
bound on the proximity measure of the $\nN_2$ neighborhood of the iterate 
after the quasi-Newton step  
\begin{equation}
  \label{n2:n2bnd1}
  \begin{split}
    \| \left( X^{k + 1} + \right. & \left. \bar\alpha \dXkm1 \right)
    \left( Z^{k + 1} + \bar\alpha \dSkm1 \right) e - \mu(\bar\alpha) e \| = \\
    & = \| \left\{ %
      (1 - \bar \alpha) \left( x^{k + 1}_i z^{k + 1}_i - \mu_{k + 1} \right) 
      + \left[ \bar \alpha_k \dxk + \bar \alpha \dxkm1 \right]_i 
        \left[ \bar \alpha_k \dsk + \bar \alpha \dskm1 \right]_i \right. \\
    & \ \qquad \left. - (1 + \bar \alpha \gamma_1) 
        {\bar \alpha_k}^2 [\dxk]_i [\dsk]_i \right\}_{i = 1}^n \| \\
    & \le (1 - \bar \alpha) \| X^{k + 1} Z^{k + 1} e - \mu_{k + 1} e
    \| + \| \left( \bar \alpha_k \dXk + \bar \alpha \dXkm1 \right)
    \left( \bar \alpha_k \dSk + \bar \alpha \dSkm1 \right) e \| \\
    & \quad + | 1 + \bar \alpha \gamma_1 | {\bar \alpha_k}^2 \| \dXk \dSk e\| \\
    & \le (1 - \bar \alpha) \theta_{k + 1} \mu_{k + 1} 
          + | 1 + \bar \alpha \gamma_1 | {\bar \alpha_k}^2 
            \frac{\theta_k^2 + n (1 - \sigma_k)^2}{2^{3/2} (1 - \theta_k)} \mu_k \\
    & \quad + \| \left( \bar \alpha_k \dXk + \bar \alpha \dXkm1 \right) 
                 \left( \bar \alpha_k \dSk + \bar \alpha \dSkm1 \right) e \|.
  \end{split}
\end{equation}
In the last inequality we used the bound on the error in the Newton 
step $\| \dXk \dSk e \|$, see~\cite[Lemma 5.4]{Wright1997}.

To further exploit (\ref{n2:n2bnd1}) we need bounds on two terms which
appear in it: $1 + \bar \alpha \gamma_1$ and the second-order error
contributed by the composite direction
$\| \left( \bar \alpha_k \dXk + \bar \alpha \dXkm1 \right) \left( \bar
  \alpha_k \dSk + \bar \alpha \dSkm1 \right) e \|$.  The following
technical result delivers a bound for $|\gamma_1|$. (Observe that
$\gamma_1$ is evaluated only when the quasi-Newton iteration is
performed.)

\begin{lemma}
  \label{l:n2gbnd}
  Let $k + 1$ be a quasi-Newton iteration of
  Algorithm~\ref{alg:pdqnipm}.  Suppose that $v$ in
  Lemma~\ref{l:bbroyd} is given by the right-hand side
  of~\eqref{def:newtonsys} and Assumption~\ref{a:n2} holds. If
  $\sszk \in (0, 1]$ and $\sigma_k \in [0, 1)$, then
  \[ 
    |\gamma_1| \le \frac{2 (1 - \sszk (1 - \sigma_k)) \sqrt{\theta_{k
          + 1}^2 + (1 - \sigma_{k + 1})^2 n}}{\sszk (1 - \sigma_k)},
  \]
  where $\gamma_1$ is defined in Lemma~\ref{l:bbroyd}.
\end{lemma}

\begin{proof}
  We use the assumptions of the lemma, the fact that the iterates are
  primal and dual feasible, the property
  $e^T (\mu_{k + 1} e - X^{k + 1} Z^{k + 1} e) = 0$ and the relation
  $\mu_{k + 1} = (1 - \bar \alpha_k (1 - \sigma_k)) \mu_{k}$ to derive
  the result
  \begin{equation*} \label{l:n2gbnd:e1} %
    \begin{split}
      \|v\| & = \| \sigma_{k + 1} \mu_{k + 1} e - X^{k + 1} Z^{k +
        1} e\| = \| (\mu_{k + 1} e - X^{k + 1} Z^{k + 1} e) - (1 -
      \sigma_{k + 1}) \mu_{k + 1} e \| \\
      & = \sqrt{\| \mu_{k + 1} e - X^{k + 1} Z^{k + 1} e \|^2 - 2 (1
        - \sigma_{k + 1}) \mu_{k + 1} e^T (\mu_{k + 1} e - X^{k + 1}
        Z^{k + 1} e) + \| (1 - \sigma_{k + 1}) \mu_{k + 1} e \|^2} \\
      & \le \sqrt{\theta_{k + 1}^2 \mu_{k + 1}^2 + (1 - \sigma_{k +
          1})^2 \mu_{k + 1}^2 n}
      \ = (1 - \sszk (1 - \sigma_k)) \mu_{k} \sqrt{\theta_{k + 1}^2 +
        (1 - \sigma_{k + 1})^2 n}\ .
    \end{split}
  \end{equation*}

  By defining $\rho_k = 1 - \sigma_k$ we can see that
  $\mu_{k + 1} = (1 - (1 - \sigma_k) \sszk) \mu_k = (1 - \rho_k \sszk)
  \mu_k$ which ensures the sufficient decrease condition of
  Lemma~\ref{l:ykbnd}. Since $\sszk > 0$ and $\sigma_k < 1$, by the
  assumptions of the lemma, we have that $\| y_k \| > 0$. Then, by
  simple substitution of the previous equation and Lemma~\ref{l:ykbnd}
  in~\eqref{projection} we have
  \begin{equation*}
    \begin{split}
      |\gamma_1| 
        \le \frac{\| v_k \|}{\| y_k \|} 
      & \le \frac{(1 - \sszk (1 - \sigma_k)) \mu_{k}
            \sqrt{\theta_{k + 1}^2 + (1 - \sigma_{k + 1})^2 n}}{\frac{1 - \sigma_k}{2} \sszk \mu_k} 
        = \frac{2 (1 - \sszk (1 - \sigma_k))
          \sqrt{\theta_{k + 1}^2 + (1 - \sigma_{k + 1})^2 n}}{\sszk (1 - \sigma_k)}.
    \end{split}
  \end{equation*}
\end{proof}

Next we turn our attention to the error in the composite direction 
$\| \left( \bar \alpha_k \dXk + \bar \alpha \dXkm1 \right) 
    \left( \bar \alpha_k \dSk + \bar \alpha \dSkm1 \right) e \|$ 
and start from a technical result.

\begin{lemma}
  \label{l:combined}
  If $k + 1$ is a quasi-Newton iteration of Algorithm~\ref{alg:pdqnipm}, 
  then
  \begin{equation}
    \label{l:combined:e}
    \begin{split}
      Z^k \left( \sszk \dXk + \ssz \dXkm1 \right) e & + X^k \left(
        \sszk \dSk + \ssz \dSkm1 \right) e = \\
      & = (\ssz + \sszk (1 - \ssz)) \left(\mu_k e - X^k Z^k e \right) +
      (\mu(\ssz) - \mu_k) e - (1 + \gamma_1) \ssz \sszk^2 \dXk \dSk e.
  \end{split}
\end{equation}
\end{lemma}

\begin{proof}
  We use equations~\eqref{def:newtonsys}, \eqref{useful-result} 
  and~\eqref{CProds-kp1} and some simple manipulations to obtain
  \begin{equation}
    \label{l:combined:e1}
    \begin{split}
      Z^k & \left( \sszk \dXk \! + \! \ssz \dXkm1 \right) e 
        + X^k \left( \sszk \dSk \! + \! \ssz \dSkm1 \right) e 
        = \sszk \left( Z^k \dxk \! + \! X^k \dsk \right) 
          + \ssz \left( Z^k \dxkm1 \! + \! X^k \dskm1 \right) \\
      & = \sszk (\sigma_k \mu_k e \! - \! X^k Z^k e) 
          + \ssz \left(\sigma_{k + 1} \mu_{k + 1} e - X^{k + 1} Z^{k + 1} e 
          - \gamma_1 {\sszk}^2 \dXk \dSk e \right) \\
      & = \sszk (\sigma_k \mu_k e \! - \! X^k Z^k e) 
          + \ssz \left(\sigma_{k + 1} \mu_{k + 1} e 
                       - (1 \! - \! \sszk) X^k Z^k e - \sszk \sigma_k \mu_k e 
                       - {\sszk}^2 \dXk \dSk e - \gamma_1 {\sszk}^2 \dXk \dSk e \right) \\
      & = (1 - \ssz) \sszk \sigma_k \mu_k e 
          - (\ssz + \sszk (1 - \ssz)) X^k Z^k e 
          + \ssz \sigma_{k + 1} \mu_{k + 1} e 
          - (1 +\gamma_1) \ssz \sszk^2 \dXk \dSk e.
    \end{split}
  \end{equation}
  After adding and subtracting the term $(\ssz + \sszk (1 - \ssz)) \mu_k e$ 
  we further rearrange the previous equation
  \begin{equation}
    \label{l:combined:e2}
    \begin{split}
      Z^k & \left( \sszk \dXk + \ssz \dXkm1 \right) e 
          + X^k \left( \sszk \dSk + \ssz \dSkm1 \right) e = \\
      & = (\ssz + \sszk (1 - \ssz)) \left(\mu_k e - X^k Z^k e \right)
          - (1 + \gamma_1) \ssz \sszk^2 \dXk \dSk e \\
      & \quad + ((1 - \ssz) \sszk \sigma_k 
        - (\ssz + \sszk (1 - \ssz))) \mu_k e + \ssz \sigma_{k + 1} \mu_{k + 1} e.
    \end{split}
  \end{equation}
  Then using $\mu_{k + 1} = (1 - \bar \alpha_k (1 - \sigma_k)) \mu_{k}$ 
  (which clearly holds for a step in Newton direction) and Lemma~\ref{l:decmu3} 
  which delivers a similar result for a step in quasi-Newton direction, we get: 
  \begin{equation}
    \label{l:combined:e3}
    \begin{split}
      ((1 - \ssz) \sszk \sigma_k 
      & - (\ssz + \sszk (1 - \ssz))) \mu_k e + \ssz \sigma_{k + 1} \mu_{k + 1} e 
        = (\sszk \sigma_k - \ssz \sszk \sigma_k - \ssz - \sszk + \ssz \sszk) \mu_k e 
          + \ssz \sigma_{k + 1} \mu_{k + 1} e \\
      & = \left[ - \ssz (1 - \sszk (1 - \sigma_k)) - \sszk (1 - \sigma_k) \right] \mu_k e 
          + \ssz \sigma_{k + 1} \mu_{k + 1} e \\
      & = \left[- 1 - \ssz (1 - \sszk (1 - \sigma_k)) + 1 - \sszk (1 - \sigma_k) \right] \mu_k e 
          + \ssz \sigma_{k + 1} \mu_{k + 1} e \\
      & = - \mu_k e - \ssz \mu_{k + 1} e + \mu_{k + 1} e + \ssz \sigma_{k + 1} \mu_{k + 1} e 
        \quad = \quad (1 - \ssz (1 - \sigma_{k + 1})) \mu_{k + 1} e - \mu_k e \\
      & = \mu(\ssz) e - \mu_k e.
    \end{split}
  \end{equation}
  By substituting~\eqref{l:combined:e3} in~\eqref{l:combined:e2}, 
  the desired result is obtained.
\end{proof}

\begin{lemma}
  \label{l:qnstepbnd}
  Let $k + 1$ be a quasi-Newton iteration of Algorithm~\ref{alg:pdqnipm} 
  and suppose that Assumption~\ref{a:n2} holds. Then
  \begin{equation*}
    \begin{split}
      \| \left( \right. & \left. \bar \alpha_k \dXk + \bar \alpha \dXkm1 \right) 
         \left( \bar \alpha_k \dSk + \bar \alpha \dSkm1 \right) e \| \le \\
      & \le \frac{\mu_k}{2^{3/2} (1 \! - \! \theta_k)} 
        \Bigg[ %
        [1 - (1 \! - \! \ssz (1 \! - \! \sigma_{k + 1})) (1 \! - \! \sszk (1 \! - \! \sigma_k))]^2 n 
        + \left( (\ssz + \sszk (1 - \ssz)) \theta_k + 
        |1 \! + \! \gamma_1| \ssz \sszk^2 
        \frac{\theta_k^2 + n (1 \! - \! \sigma_k)^2}{2^{3/2} (1 - \theta_k)} \right)^2 
        \Bigg].
    \end{split}
  \end{equation*}
\end{lemma}

\begin{proof}
Let us first define $D^k = (X^k)^{1/2} (Z^k)^{-1/2}$ and the scaled vectors
  \[
      u_k = {(D^k)}^{-1} \left( \bar \alpha_k \dXk + \bar \alpha \dXkm1 \right) e 
      \quad {\mbox {and}} \quad 
      v_k = {D^k} \left( \bar \alpha_k \dSk + \bar \alpha \dSkm1 \right) e.
  \]
Using Lemma~\ref{l:ortho}, Corollary~\ref{c:ortho} and observing 
that all the involved matrices are diagonal, we get $u_k^T v_k = 0$.
Hence vectors $u_k$ and $v_k$ satisfy the assumptions of Lemma~5.3 
in \cite{Wright1997}. 
With $U_k = \diag(u_k)$ and $V_k = \diag(v_k)$, we write 
\begin{equation}
  \label{l:qnstepbnd:e1}
  \begin{split}
    \| \left( \bar \alpha_k \dXk + \bar \alpha \dXkm1 \right) & 
       \left( \bar \alpha_k \dSk + \bar \alpha \dSkm1 \right) e \| 
       = \| {(D^k)}^{-1} \left( \bar \alpha_k \dXk + \bar \alpha \dXkm1 \right) 
            {D^k} \left( \bar \alpha_k \dSk + \bar \alpha \dSkm1 \right) e \| \\
    & = \| U_k V_k e \| \ \le \ 2^{-3/2} \| u_k + v_k \|^2 \\
    & = 2^{-3/2} \| {(D^k)}^{-1} \left( \bar \alpha_k \dXk + \bar \alpha \dXkm1 \right) e 
                    + {D^k} \left( \bar \alpha_k \dSk + \bar \alpha \dSkm1 \right) e \|^2. 
  \end{split}
\end{equation}
After multiplying both sides of~\eqref{l:combined:e} by $(X^k Z^k)^{-1/2}$ 
and replacing it in~\eqref{l:qnstepbnd:e1} we obtain
\begin{equation}
  \label{l:qnstepbnd:e2}
  \begin{split}
    \| & \left. \left( \bar \alpha_k \dXk \right. + \bar \alpha \dXkm1 \right) 
       \left( \bar \alpha_k \dSk + \bar \alpha \dSkm1 \right) e \| \le \\
    & \le 2^{-3/2} \| (X^k Z^k)^{-1/2} \left\{ (\ssz + \sszk (1 - \ssz))
      \left(\mu_k e - X^k Z^k e \right) + (\mu(\ssz) - \mu_k) e 
      - (1 + \gamma_1) \ssz \sszk^2 \dXk \dSk e \right\} \|^2 \\
    & = 2^{-3/2} \sum_{i = 1}^n \frac{
      \left\{ (\ssz + \sszk (1 - \ssz)) \left(\mu_k - x^k_i z^k_i \right) 
      + \mu(\ssz) - \mu_k 
      - (1 + \gamma_1) \ssz \sszk^2 [\dxk]_i [\dsk]_i \right\}^2}{x^k_i z^k_i} \\
    & \le \frac{2^{-3/2}}{(1 - \theta_k) \mu_k} \| (\ssz + \sszk (1 -
    \ssz)) \left(\mu_k e - X^k Z^k e \right) - (1 + \gamma_1) \ssz
    \sszk^2 \dXk \dSk e + (\mu(\ssz) - \mu_k) e \|^2,
  \end{split}
\end{equation}
where the last inequality comes from the fact that $\xysk{k}$ belongs
to $\nN_2(\theta_k)$, hence 
$(1 - \theta_k) \mu_k \le x^k_i z^k_i \le (1 + \theta_k) \mu_k$ 
for all $i$. Now we use Lemma~\ref{l:ortho} again and the definition 
of $\mu_k$ to observe that
\begin{equation}
  \label{useful-result-2}
  e^T \left( (\ssz + \sszk (1 - \ssz)) \left(\mu_k e - X^k Z^k e \right) 
             - (1 + \gamma_1) \ssz \sszk^2 \dXk \dSk e \right) = 0
\end{equation}
and further rearrange~\eqref{l:qnstepbnd:e2}:
\begin{equation}
  \label{l:qnstepbnd:e3}
  \begin{split}
    \| & \left. \left( \bar \alpha_k \dXk \right. + \bar \alpha \dXkm1 \right) 
                \left( \bar \alpha_k \dSk + \bar \alpha \dSkm1 \right) e \| \le \\
    & \le \frac{2^{-3/2}}{(1 - \theta_k) \mu_k} 
        \left[ \left\| (\ssz + \sszk (1 - \ssz)) 
               \left(\mu_k e - X^k Z^k e \right) 
               - (1 + \gamma_1) \ssz \sszk^2 \dXk \dSk e \right\|^2 
               + \| (\mu(\ssz) - \mu_k) e \|^2 
        \right].
  \end{split}
\end{equation}
The second norm on the right-hand side of~\eqref{l:qnstepbnd:e3} is
given by~\eqref{l:qnstepbnd:e4}.  Using the definition of
$\nN_2(\theta_k)$ neighborhood and Lemma~5.4 from~\cite{Wright1997}, 
for the first norm we get
\begin{equation}
  \label{l:qnstepbnd:e5}
  \begin{split}
    \left\| (\ssz + \sszk (1 - \ssz)) \right. & \left. \left(\mu_k e -
        X^k Z^k e \right) - (1 + \gamma_1) \ssz \sszk^2 \dXk \dSk
      e \right\|^2 \le \\
    & \le \left( \left\| (\ssz + \sszk (1 - \ssz)) \left(\mu_k e - X^k
          Z^k e \right) \right\| + \left\| (1 + \gamma_1) \ssz
        \sszk^2 \dXk \dSk e \right\| \right)^2\\
    & \le \left( %
      (\ssz + \sszk (1 - \ssz)) \theta_k + %
      |1 + \gamma_1| \ssz \sszk^2 \frac{\theta_k^2 + n (1 -
        \sigma_k)^2}{2^{3/2} (1 - \theta_k)} %
    \right)^2 \mu_k^2 .\\
  \end{split}
\end{equation}
Finally, by substituting~\eqref{l:qnstepbnd:e4} and~\eqref{l:qnstepbnd:e5}
in~\eqref{l:qnstepbnd:e3} we obtain the required result. 
\end{proof}

We are ready to state the main result of this subsection.  In
Theorem~\ref{t:n2conv} we show that it is possible to choose
sufficiently small values for the step-sizes $\sszk$ and $\ssz$ such
that $\xysk{k + 2} \in \nN_2(\theta_{k})$. Therefore, we ensure that
the quasi-Newton step taken after a Newton step remains in the $\nN_2$
neighborhood. This implies that all the iterates generated by the
algorithm belong to $\nN_2$. The upper bounds for the step-sizes as
stated in the theorem are then used to determine the worst-case
complexity of Algorithm~\ref{alg:pdqnipm} operating in $\nN_2$.
First, recall~\cite[Theorem~5.6]{Wright1997} that it is possible to
choose parameters $\theta_k, \theta_{k + 1} \in (0, 1)$ and
$\sigma_k, \sigma_{k + 1} \in (0, 1)$ so that
  \begin{equation}
    \label{t:n2conv:e2}
    \frac{\theta_k^2 + n (1 - \sigma_k)^2}{2^{3/2} (1 - \theta_k)}
       \le \theta_k \sigma_k
    \quad \text{and} \quad 
    \frac{\theta_{k + 1}^2 + n (1 - \sigma_{k + 1})^2}{2^{3/2} (1 - \theta_{k + 1})}
       \le \theta_{k + 1} \sigma_{k + 1}.
  \end{equation}

{ 
\begin{theorem}
  \label{t:n2conv}
  Let $k + 1$ be a quasi-Newton iteration of
  Algorithm~\ref{alg:pdqnipm}. Suppose that Assumptions~\ref{a:pdfeas}
  and~\ref{a:n2} hold and that $\theta_{k + 1} = \theta_k$ and
  $\sigma_{k + 1} = \sigma_k$. If the step-sizes in Newton and
  quasi-Newton iterations $\sszk$ and $\ssz$ satisfy
  \begin{equation} \label{t:n2conv:ss}
    \sszk \in \left( 0, \min\left\{\frac{1 - \sigma_k}{4 \sigma_k},
        \frac{\sigma_k (1 - \sigma_k)}{10 (1 - \sigma_k) + 4}\right\}
    \right] %
    \quad \text{and} \quad %
    \ssz \in \left[ \sszk, \frac{\sigma_k (1 - \sigma_k)}{10 (1 -
        \sigma_k) + 4} \right]
  \end{equation}
  for $\sigma_k \in [0, 1)$ and $\theta_k\in (0, 16/25)$, then
  \[
    \xysk{k + 2} \doteq \xysk{k + 1} + \ssz \dxyskm1 \in \nN_2(\theta_{k}).
  \]
\end{theorem}

\begin{proof}
  We will first show that for all $\sszk$ and $\ssz$ satisfying 
  the conditions of the theorem
  \begin{equation}
    \label{t:n2conv:e11}
    \| \left( X^{k + 1} + \right. \left. \bar\alpha \dXkm1 \right)
    \left( Z^{k + 1} + \bar\alpha \dSkm1 \right) e - \mu(\bar\alpha) e
    \| - \theta_{k} \mu(\ssz) \le 0
  \end{equation}
  holds.
  This condition ensures that the Newton step at iteration $k + 2$
  will be successfully performed and the algorithm converges. By
  inequality~\eqref{n2:n2bnd1}, condition~\eqref{t:n2conv:e11} is
  satisfied if
  \begin{equation} \label{t:n2conv:new2}
    \begin{split}
      \left[(1 - \bar \alpha) \theta_{k + 1} \mu_{k + 1} - \theta_{k}
      \mu(\ssz)\right] & + | 1 + \bar \alpha \gamma_1 | {\bar \alpha_k}^2
      \frac{\theta_k^2 + n (1 - \sigma_k)^2}{2^{3/2} (1 - \theta_k)}
      \mu_k \\
      & + \| \left( \bar \alpha_k \dXk + \bar \alpha \dXkm1 \right)
      \left( \bar \alpha_k \dSk + \bar \alpha \dSkm1 \right) e \| \le
      0.
    \end{split}
  \end{equation}
  We will derive bounds to each term on the left-hand side of this
  inequality in order to find an expression in a form
  $K_1 \ssz^2 - K_2 \ssz$, $K_1, K_2 > 0$, which will be nonpositive
  for small values of $\ssz$.

  For the first term, we use the fact that
  $\theta_{k + 1} = \theta_k$, $\sigma_{k + 1} = \sigma_k$ and
  $\sszk \in (0, 1)$. In addition, we use Lemma~\ref{l:decmu3} to
  expand $\mu(\ssz)$ and the fact that $\mu_{k + 1}$ was calculated in
  the Newton step $k$ to obtain
  \begin{equation}
    \label{t:n2conv:new1}
    (1 - \bar \alpha) \theta_{k + 1} \mu_{k + 1} - \theta_{k}
      \mu(\ssz) = [ (1 - \ssz) (\theta_{k + 1} - \theta_{k}) - \ssz
    \sigma_{k + 1} \theta_{k} ] (1 - \sszk (1 - \sigma_k)) \mu_k
    \le - \ssz \sigma_k^2 \theta_{k} \mu_k.
  \end{equation}
  For the second and third terms, we first apply~\eqref{t:n2conv:e2}
  in Lemma~\ref{l:n2gbnd} to simplify the bound of $\gamma_1$:
  \begin{equation} \label{t:n2conv:new4}
    \begin{split}
      |\gamma_1| & \le \frac{2 (1 - \sszk (1 - \sigma_k))
        \sqrt{2^{3/2} (1 - \theta_{k + 1})} \sqrt{\frac{\theta_{k +
              1}^2 + (1 - \sigma_{k + 1})^2 n}{2^{3/2}
            (1 - \theta_{k + 1})}}}{\sszk (1 - \sigma_k)} \\
      & \le \frac{2 (1 - \sszk (1 - \sigma_k)) \sqrt{2^{3/2} (1 -
          \theta_{k + 1})} \sqrt{\theta_{k + 1} \sigma_{k + 1}}}{\sszk
        (1 - \sigma_k)} \ \le\ \frac{4}{\sszk (1 - \sigma_k)} - 4.
    \end{split}
  \end{equation}
  Therefore, using~\eqref{t:n2conv:e2} again, we derive the following 
  bound to the second term
  \begin{equation} \label{t:n2conv:new7}
    \begin{split}
      |1 + \ssz \gamma_1| \sszk^2 & \frac{\theta_k^2 + n (1 -
        \sigma_k)^2}{2^{3/2} (1 - \theta_k)} \mu_k \le \left[ 1 +
        \ssz \left( \frac{4}{\sszk (1 - \sigma_k)} - 4 \right)
      \right] \sszk^2 \theta_k \sigma_k \mu_k \\
      & = \sszk^2 \theta_k \sigma_k \mu_k + 4 \ssz \sszk \left(
        \frac{1 - \sszk (1 - \sigma_k)}{1 - \sigma_k} \right) \theta_k
      \sigma_k \mu_k \ \le\ \ssz^2 \left[ 1 + \frac{4}{1 -
          \sigma_k} \right] \theta_k \sigma_k \mu_k,
    \end{split}
  \end{equation}
  where in the last inequality we used the condition
  $\sszk \le \ssz \le 1$ from~\eqref{t:n2conv:ss}.  For the
  third term in~\eqref{t:n2conv:new2}, we use the bound obtained in
  Lemma~\ref{l:qnstepbnd} and analyze each part of it
  independently. First, since $\sszk \le \ssz$ and
  $\sigma_{k + 1} = \sigma_k$, we observe that
  \begin{equation} \label{t:n2conv:new3}
    \begin{split}
      [1 - (1 - \ssz (1 - \sigma_{k + 1})) & (1 - \sszk (1 - \sigma_k))]^2 n 
      = [\sszk (1 - \sigma_k) + \ssz (1 - \sszk (1 - \sigma_{k})) (1 - \sigma_k))]^2 n \\
      & \le \ssz^2 (1 - \sigma_k)^2 [1 + (1 - \sszk (1 - \sigma_k))]^2 n 
        \le 4 \ssz^2 (1 - \sigma_k)^2 n.
    \end{split}
  \end{equation}
  Using bound~\eqref{t:n2conv:new4}, assumption 
  $\sszk \le (1 - \sigma_k) / (4 \sigma_k)$ in~\eqref{t:n2conv:ss}, 
  and~\eqref{t:n2conv:e2} again, we also obtain
  \begin{equation}
    \label{t:n2conv:new5}
    \begin{split}
      \left[ %
        (\ssz + \sszk (1 - \ssz)) \theta_k + %
        |1 + \gamma_1| \ssz \sszk^2 \frac{\theta_k^2 + n (1 -
          \sigma_k)^2}{2^{3/2} (1 - \theta_k)} %
      \right]^2 & \le \left[\ssz + \sszk (1 - \ssz) +
        \left(\frac{4}{\sszk (1 - \sigma_k)} - 3 \right) \ssz \sszk^2
        \sigma_k \right]^2 \theta_k^2 \\
      & \le \left[\ssz + \sszk (1 - \ssz) + \frac{4}{\sszk (1 -
          \sigma_k)} \ssz \sszk^2 \sigma_k \right]^2 \theta_k^2 \\
      & \le \ssz^2 \left[ 1 + (1 - \ssz) + \frac{4 \sszk \sigma_k}{(1
          - \sigma_k)} \right]^2 \theta_k^2 \\
      & \le 9 \ssz^2 \theta_k^2.
    \end{split}
  \end{equation}
  By combining~\eqref{t:n2conv:new3} and~\eqref{t:n2conv:new5} in the
  statement of Lemma~\ref{l:qnstepbnd} and
  applying~\eqref{t:n2conv:e2} once more, we derive a bound to the
  third term of~\eqref{t:n2conv:new2}
  \begin{equation}
    \label{t:n2conv:new6}
    \| \left( \bar \alpha_k \dXk + \bar \alpha \dXkm1 \right)
    \left( \bar \alpha_k \dSk + \bar \alpha \dSkm1 \right) e \|
    \le \frac{\mu_k}{2^{3/2} (1 - \theta_k)}
    \left(4 \ssz^2 (1 - \sigma_k)^2 n + 9 \ssz^2
      \theta_k^2 \right) \le 9 \ssz^2 \theta_k \sigma_k \mu_k.
  \end{equation}
  By~\eqref{t:n2conv:new1}, \eqref{t:n2conv:new7}
  and~\eqref{t:n2conv:new6}, the following bound on expression 
  in~\eqref{t:n2conv:new2} is obtained 
  \begin{equation}
    \label{t:n2conv:new8}
    \begin{split}
      (1 - \ssz) \theta_{k + 1} \mu_{k + 1} - \theta_{k} \mu(\ssz)
      & + | 1 + \ssz \gamma_1 | \sszk^2 \frac{\theta_k^2 + n (1 -
        \sigma_k)^2}{2^{3/2} (1 - \theta_k)}
      \mu_k \\
      & + \| \left( \sszk \dXk + \ssz \dXkm1 \right)
      \left( \sszk \dSk + \ssz \dSkm1 \right) e \| \\
      & \le - \ssz \sigma_k^2 \theta_{k} \mu_k + \ssz^2 \left[ 1 +
        \frac{4}{1 - \sigma_k} \right]
      \theta_k \sigma_k \mu_k + 9 \ssz^2 \theta_k \sigma_k \mu_k \\
      & \le \ssz \left[ 10 \ssz + \frac{4 \ssz}{1 - \sigma_k} -
        \sigma_k \right] \theta_k \sigma_k \mu_k,
    \end{split}
  \end{equation}
  which is negative only if
  $\ssz \le \sigma_k (1 - \sigma_k) / (10 (1 - \sigma_k) + 4)$, as
  requested by~\eqref{t:n2conv:ss}. This bound on $\ssz$
  implies a bound on $\sszk$, due to condition $\sszk \le
  \ssz$. Therefore, we arrive in the step-size conditions~\eqref{t:n2conv:ss} of the
  theorem.
  Using~\eqref{t:n2conv:e11} we also have that
  \begin{equation} \label{t:n2conv:e12} %
    \left( x^{k + 1}_i + \ssz [\dxkm1]_i \right) 
    \left( z^{k + 1}_i + \ssz [\dskm1]_i \right) 
    \ge (1 - \theta_{k}) \mu(\ssz) > 0.
  \end{equation}
  
  We now show that the new iterate belongs to $\mathcal{F}$. By
  Assumption~\ref{a:pdfeas} and Lemma~\ref{l:bbroyd} we know that all
  iterates remain primal and dual feasible. It remains to show that
  $x^{k + 2}$ and $z^{k + 2}$ are strictly positive. We follow the
  same arguments as~\cite{Monteiro1989} and adapt them to our
  case. Suppose by contradiction that $x^{k + 2}_i \le 0$ or
  $z^{k + 2}_i \le 0$ hold for some $i$. By~\eqref{t:n2conv:e12}, we
  have that $x^{k + 2}_i < 0$ and $z^{k + 2}_i < 0$ and that implies
  $x^k_i z^k_i < (\sszk [\dxk]_i + \ssz [\dxkm1]_i) (\sszk [\dsk]_i +
  \ssz [\dskm1]_i) \le 9 \ssz^2 \theta_k \sigma_k \mu_k$ by
  inequality~\eqref{t:n2conv:new6}.  Since
  $\xysk{k} \in \nN_2(\theta_{k})$ and $\ssz \le 1/4$,
  by~\eqref{t:n2conv:ss}, we conclude that
  $(1 - \theta_k) \mu_k < (9/16) \theta_k \mu_k$ and, therefore,
  $\theta_k > 16/25$, which contradicts the choice of $\theta_k$.
  Hence, $\xysk{k + 2}$ belongs to $\nN_2(\theta_{k})$ and the Newton
  step at iteration $k + 2$ also falls in the $\nN_2(\theta_k)$
  neighborhood.
\end{proof}

For the polynomial convergence of Algorithm~\ref{alg:pdqnipm}, we
define $\sigma_k = \sigma_{k + 1} = 1 - 0.4 / \sqrt{n}$ and
$\theta_{k} = \theta_{k + 1} = \theta_{k + 2} = 0.4$, which satisfy
condition~\eqref{t:n2conv:e2}~\cite{Wright1997} and maintain all the
previous results. By Lemma~\ref{l:decmu3},
$\mu(\ssz) \le \mu_{k + 1} \le \mu_{k}$.  So it is enough to look just
at the Newton steps, which are easier to
analyze. Using~\eqref{t:n2conv:ss} it is not hard to see that
\[
  \sszk \ge \min\left\{\frac{1 - \sigma_k}{4 \sigma_k}, \frac{\sigma_k
      (1 - \sigma_k)}{10 (1 - \sigma_k) + 4}\right\} \ge
  \min\left\{\frac{0.1}{\sqrt{n}}, \frac{0.03}{\sqrt{n}} \right\} =
  \frac{0.03}{\sqrt{n}}.
\]
Therefore,
\[
    \mu_{k + 1} \le \left(1 - \frac{0.03}{\sqrt{n}} \frac{0.4}{\sqrt{n}} \right) \mu_k 
    = \left(1 - \frac{0.012}{n} \right) \mu_k, \quad k = 0, 2, 4, \dots,
\]
from which the convergence with the worst-case iteration 
complexity of $O(n)$ can be derived.
}

\subsection{The $\nN_s$ neighborhood}
\label{ns}

Colombo and Gondzio~\cite{Colombo2008} used the symmetric
neighborhood $\nN_s(\gamma)$, defined by
\[
  \nN_s(\gamma) \doteq \left\{ (x, \lambda, z) \in \F\ |\ \gamma \mu \le
  x_i z_i \le \frac{1}{\gamma} \mu, i = 1,\dots, n \right\},
\]
for $\gamma \in (0, 1)$, which is related with the
$\nN_{-\infty}(\gamma)$ neighborhood used in long-step primal-dual
interior point algorithms. The idea of the symmetric neighborhood is
to add an upper bound on the complementarity pairs, so that their 
products do not become too large with respect to the average. 
The authors showed that the worst-case iteration complexity for linear 
feasible primal-dual interior point methods remains $O(n)$ and 
the new neighborhood has a better practical interpretation. 
As \hopdm~\cite{Gondzio1995} implements the $\nN_s$ neighborhood and 
it was used in the numerical experiments in~\cite{GondzioSobral2019} 
for quasi-Newton IPM, it is natural to ask about the iteration
complexity of Algorithm~\ref{alg:pdqnipm} operating in the $\nN_s$ 
neighborhood. 
The analysis presented below will follow closely that from Subsection~\ref{n2}. 
We start from an assumption, but it is worth observing that, 
from~\cite{Colombo2008,Wright1997}, this assumption holds if the step-size 
in the Newton direction is sufficiently small: 
   $\sszk \in \left[0, \min \left\{ 2^{3/2} \frac{1 - \gamma}{1 + \gamma}
    \frac{\sigma_k}{n}, 2^{3/2} \gamma \frac{1 - \gamma}{1 + \gamma}
    \frac{\sigma_k}{n} \right\} \right]$. 

\begin{assumption}
  \label{a:ns}
  Let $\gamma \in (0, 1)$ and $\xysk{k} \in \nN_s(\gamma)$. 
  Let the iterate after a step $\sszk$ in Newton direction 
  also satisfy $\xysk{k + 1} \in \nN_s(\gamma)$. 
\end{assumption}

Our main goal is to show that the next iterate obtained after 
a step in the quasi-Newton direction 
\[
  \xysk{k + 2} = \xysk{k + 1} + \bar \alpha \dxyskm1
\]
also belongs to $\nN_s(\gamma)$ if suitable step-sizes 
$\bar \alpha$ and $\bar \alpha_k$ are chosen. 
To demonstrate this, we will consider lower and upper bounds 
on the complementarity products in the $\nN_s(\gamma)$ neighborhood 
using two possible values of $\zeta \in \{\gamma, \frac{1}{\gamma}\}$.
We start the analysis from expanding the complementarity product 
at the quasi-Newton iteration
\begin{equation}
  \label{ns:term}
  \left[ (X^{k + 1} + \bar \alpha \dXkm1) 
         (Z^{k + 1} + \bar \alpha \dSkm1) e 
  \right]_i - \zeta \mu(\ssz),
\end{equation}
where $\ssz \in (0, 1)$ is the step-size associated with the quasi-Newton
direction. To guarantee that the new iterate belongs to $\nN_s(\gamma)$, 
for $\zeta = \gamma$ the expression in \eqref{ns:term} should 
be non-negative and for $\zeta = 1/\gamma$ it should be non-positive, 
for $i = 1, \dots, n$. 
Using~\eqref{singleProd} and Lemma~\ref{l:decmu3}, 
we rewrite the expression in~\eqref{ns:term} 
\begin{equation}
  \label{ns:e1}
  \begin{split}
    \left[ x^{k + 1} \right. & + \left. \bar \alpha \dxkm1 \right]_i
    \left[ z^{k + 1} + \bar \alpha \dskm1 \right]_i - \zeta \mu(\ssz) 
      \quad = \quad (1 - \ssz) (x_i^{k + 1} z_i^{k + 1} - \zeta \mu_{k + 1}) 
        + (1 - \zeta) \ssz \sigma_{k + 1} \mu_{k + 1} \\
    & \quad + \left[\sszk \dxk + \ssz \dxkm1 \right]_i
              \left[\sszk \dsk + \ssz \dskm1 \right]_i 
            - (1 + \gamma_1 \ssz) \sszk^2 [\dxk]_i [\dsk]_i.
  \end{split}
\end{equation}

To deliver the main result of this section we will need a bound 
for the quasi-Newton term $\gamma_1$ defined in Lemma~\ref{l:bbroyd} 
when the algorithm operates in the $\nN_s(\gamma)$ neighborhood.

\begin{lemma}
  \label{l:nsgbnd}
  Let $k + 1$ be a quasi-Newton iteration of
  Algorithm~\ref{alg:pdqnipm} operating in the $\nN_s(\gamma)$
  neighborhood.  Suppose that $v$ in Lemma~\ref{l:bbroyd} is given by
  the right-hand side of~\eqref{def:newtonsys} and
  Assumption~\ref{a:ns} holds. If $\sszk \in (0, 1]$
  and $\sigma_k \in [0, 1)$, then
  \[
    |\gamma_1| \le \frac{2 \sqrt{n}}{(1 - \sigma_k) \sszk \gamma}.
  \]
\end{lemma}

\begin{proof}
  Using the conditions of the lemma, the definition of $\mu$ and the
  fact that $\sigma_{k + 1} \in [0, 1]$, we have that
  \begin{equation*}
    \begin{split}
      \| v \| & = \| \sigma_{k + 1} \mu_{k + 1} e - X^{k + 1} Z^{k + 1} e \| 
                = \sqrt{\sigma_{k + 1}^2 \mu_{k + 1}^2 n 
                        - 2 \sigma_{k + 1} \mu_{k + 1}^2 n 
                        + \sum_{i = 1}^n (x_i^{k + 1} z_i^{k + 1})^2} \\
      & \le \sqrt{(1 / \gamma^2) \mu_{k + 1}^2 n 
                  - (2 - \sigma_{k + 1}) \sigma_{k + 1} \mu_{k + 1}^2 n} 
        \le \sqrt{(1/\gamma^2) - \sigma_{k + 1}} \sqrt{n} \mu_{k + 1} \\
      & \le \frac{\sqrt{n}}{\gamma} \mu_{k + 1} 
        = \frac{(1 - \sszk (1 - \sigma_k)) \sqrt{n}}{\gamma} \mu_k 
        \le \frac{\sqrt{n}}{\gamma} \mu_k.
    \end{split}
  \end{equation*}
  Since $\sszk \in (0, 1]$ and $\sigma_k \in [0, 1)$, by defining
  $\rho_k = 1 - \sigma_k > 0$ we can again use
  $\mu_{k + 1} = (1 - \sszk (1 - \sigma_k)) \mu_k$ and
  Lemma~\ref{l:ykbnd} to ensure that $\| y_k \| > 0$. Therefore,
  by~\eqref{projection}, Lemma~\ref{l:ykbnd} and the previous result,
  \[
    |\gamma_1|
    \le \frac{\| v \|}{\|y_k\|} 
    \le \frac{(1 - \sszk (1 - \sigma_k)) \sqrt{n} \mu_k}{\gamma} 
        \frac{2}{(1 - \sigma_k) \sszk \mu_k} 
    =   \frac{2 (1 - \sszk (1 - \sigma_k)) \sqrt{n}}{(1 - \sigma_k) \sszk \gamma} 
    \le \frac{2 \sqrt{n}}{(1 - \sigma_k) \sszk \gamma}.
  \]
\end{proof}

The next lemma delivers a bound for term
$\left[\sszk \dxk + \ssz \dxkm1 \right]_i \left[\sszk \dsk + \ssz \dskm1 \right]_i$ 
under Assumption~\ref{a:ns}. Most of the calculations have already been made 
in the proof of Lemma~\ref{l:qnstepbnd}. 
%
%
Let us mention that Lemmas~\ref{l:qnstepbnd} and~\ref{l:ns:qnstepbnd} can be 
viewed as the quasi-Newton versions of~\cite[Lemma 5.10]{Wright1997}.

\begin{lemma}
  \label{l:ns:qnstepbnd}
  Let $k + 1$ be a quasi-Newton iteration of Algorithm~\ref{alg:pdqnipm} 
  and suppose that Assumption~\ref{a:ns} holds. Then,
    \begin{equation*}
    \begin{split}
      \| \left( \right. & \left. \bar \alpha_k \dXk + \bar \alpha \dXkm1 \right) 
         \left( \bar \alpha_k \dSk + \bar \alpha \dSkm1 \right) e \| \le \\
      &  \le \frac{n \mu_k}{2^{3/2} \gamma} 
        \Bigg\{ %
          \left[ (\ssz + \sszk (1 - \ssz)) 
             + \frac{|1 + \gamma_1| \ssz \sszk^2}{2^{3/2}} \right]^2 
          \left( \frac{1 + \gamma}{\gamma} \right)^2 n 
             + [1 - (1 - \ssz (1 - \sigma_{k + 1})) (1 - \sszk (1 - \sigma_k))]^2 
        \Bigg\}.
  \end{split}
\end{equation*}
\end{lemma}

\begin{proof}
  We observe that many arguments used at the beginning of the proof 
  of Lemma~\ref{l:qnstepbnd} remain valid for the $\nN_s$ neighborhood. 
  (Only the bound $x_i^k z_i^k \geq (1 - \theta_k) \mu_k$ needs 
  to be replaced with $x_i^k z_i^k \geq \gamma \mu_k$.) 
  Then inequalities~(\ref{l:qnstepbnd:e2}) and~(\ref{l:qnstepbnd:e3})
  are replaced with the following
  \begin{equation}
    \label{l:ns:qnstepbnd:e1}
    \begin{split}
      \| \left( \bar \alpha_k \dXk \right. + & \left. \bar \alpha
        \dXkm1 \right) \left( \bar \alpha_k \dSk + \bar \alpha \dSkm1
      \right) e \| \le \\
      & \le 2^{-3/2} \sum_{i = 1}^n \frac{ %
        \left\{ (\ssz + \sszk (1 - \ssz)) \left(\mu_k - x^k_i z^k_i
          \right) + \mu(\ssz) - \mu_k - (1 + \gamma_1) \ssz \sszk^2
          [\dxk]_i [\dsk]_i \right\}^2 %
      }{x^k_i z^k_i} \\
      & \le \frac{1}{2^{3/2} \gamma \mu_k} \| (\ssz + \sszk (1 -
      \ssz)) \left(\mu_k e - X^k Z^k e \right) + (\mu(\ssz) - \mu_k) e
      - (1 + \gamma_1) \ssz \sszk^2 \dXk \dSk e \|^2 \\
      & = \frac{1}{2^{3/2} \gamma \mu_k} \left[\| (\ssz + \sszk (1 -
        \ssz)) \left( \mu_k e - X^k Z^k e \right) - (1 + \gamma_1)
        \ssz \sszk^2 \dXk \dSk e \|^2 + \| (\mu(\ssz) - \mu_k) e \|^2
      \right],
    \end{split}    
  \end{equation}
  where in the last equality we have used equation~\eqref{useful-result-2}.
  We already have the expression for $\| (\mu(\ssz) - \mu_k) e \|^2$ 
  (see~\eqref{l:qnstepbnd:e4}), but we need a bound for the first norm 
  in~\eqref{l:ns:qnstepbnd:e1}. We observe that
  \begin{equation}
    \label{l:ns:qnstepbnd:e2}
    \begin{split}
      \| \mu_k e - X^k Z^k e \| 
      & = \sqrt{\sum_{i = 1}^n (x_i^k z_i^k)^2 - 2 \mu_k^2 n + \mu_k^2 n} 
          \le \sqrt{ \frac{\mu_k^2 n}{\gamma^2} - \mu_k^2 n} 
          \le \frac{1}{\gamma} \sqrt{n} \mu_k 
          \le \frac{1 + \gamma}{\gamma} \sqrt{n} \mu_k.
    \end{split}
  \end{equation}
  Therefore, by~\eqref{l:ns:qnstepbnd:e2} and the bound on
  $\left\| \dXk \dSk e \right\|$ obtained in~\cite[Lemma
  5.10]{Wright1997} (which also holds for $\nN_s$~\cite{Colombo2008})
  \begin{equation}
    \label{l:ns:qnstepbnd:e3}
    \begin{split}
      \left\| (\ssz + \sszk (1 - \ssz)) \left(\mu_k e - X^k Z^k e
        \right) \right. - & \left. (1 + \gamma_1) \ssz \sszk^2 \dXk
        \dSk e \right\|^2 \le \\
      & \le \left[ (\ssz + \sszk (1 - \ssz)) \left\| \mu_k e - X^k Z^k
          e \right\| + |1 + \gamma_1| \ssz \sszk^2
        \left\| \dXk \dSk e \right\| \right]^2 \\
      & \le \left[ (\ssz + \sszk (1 - \ssz)) \left( \frac{1 +
            \gamma}{\gamma} \right) \sqrt{n} \mu_k + \frac{|1 +
          \gamma_1| \ssz \sszk^2}{2^{3/2}} \left( \frac{1 +
            \gamma}{\gamma} \right) \mu_k n \right]^2 \\
      & \le \left[ (\ssz + \sszk (1 - \ssz)) + \frac{|1 + \gamma_1|
          \ssz \sszk^2}{2^{3/2}} \right]^2 \left( \frac{1 +
          \gamma}{\gamma} \right)^2 n^2 \mu_k^2,
    \end{split}
  \end{equation}
  since $n \ge 1$. Using~\eqref{l:qnstepbnd:e4} and~\eqref{l:ns:qnstepbnd:e3} 
  in~\eqref{l:ns:qnstepbnd:e1}, we obtain the desired result. 
%
%
\end{proof}

Using~\eqref{ns:e1} and the bounds obtained so far, we now show that,
for sufficiently small step-sizes $\sszk$ and $\ssz$, in the Newton and
quasi-Newton iterations, respectively, the point $\xysk{k + 2}$ also
belongs to $\nN_s(\gamma)$. 
The upper bounds for the step-sizes delivered by the theorem below are 
then used to determine the $O(n^3)$ iteration worst-case complexity 
of Algorithm~\ref{alg:pdqnipm} operating in $\nN_s(\gamma)$.

\begin{theorem}
  \label{t:nsconv}
  Suppose that $k + 1$ is a quasi-Newton iteration of
  Algorithm~\ref{alg:pdqnipm} and Assumption~\ref{a:ns} holds. 
  Define
  \[
    l = \frac{\frac{1}{2} \sigma_{min}}
          {\frac{3}{2^{3/2} \gamma} \left( 2 + \frac{1}{\gamma (1 - \sigma_{max})} \right)^2 
          \left( \frac{1 + \gamma}{\gamma} \right)^2},
  \]
  where $0 < \sigma_{min} \le \sigma_k \le \sigma_{max} < 1$ 
  for all $k = 0,1,2, \dots$ If
  \begin{equation}
    \label{t:nsconv:bnds}
    \sszk \in \left(0, \frac{(1 - \gamma) l}{2 n^3} \right]\quad
    \text{and}\quad \ssz \in \left[2 \sszk, \frac{(1 - \gamma) l}{n^3}
    \right]
  \end{equation}
  then
  $\gamma \mu(\ssz) \le x^{k + 2}_i z^{k + 2}_i \le (1 / \gamma)
    \mu(\ssz)$ for all $i = 1, \dots, n$.
  If, in addition, $\gamma \ge \sigma_{min} / 4$, then
    $\xysk{k + 2} \in \nN_s(\gamma)$.
\end{theorem}

\begin{proof}
  By construction we guarantee that $\ssz \ge 2 \sszk$ as needed 
  in~\eqref{t:nsconv:bnds}.
  We start the proof by setting $\xi = \gamma$ in~\eqref{ns:e1} and
  showing that
  $\left[ x^{k + 1} + \bar \alpha \dxkm1 \right]_i 
   \left[ z^{k + 1} + \bar \alpha \dskm1 \right]_i - \gamma \mu(\ssz) \ge 0$ 
  for sufficiently small step-sizes. By~\eqref{ns:e1}, using
  Assumption~\ref{a:ns}, Lemma~\ref{l:ns:qnstepbnd} 
  and~\cite[Lemma 5.10]{Wright1997}, we obtain 
  \begin{equation}
    \label{t:nsconv:e1}
    \begin{split}
      \left[ x^{k + 1} \right. & + \left. \bar \alpha \dxkm1 \right]_i
      \left[ z^{k + 1} + \bar \alpha \dskm1 \right]_i - \gamma
      \mu(\ssz) = \\
      & = (1 - \ssz) (x_i^{k + 1} z_i^{k + 1} - \gamma \mu_{k + 1}) +
      (1 - \gamma) \ssz \sigma_{k + 1} \mu_{k + 1} \\
      & \quad + \left[\sszk \dxk + \ssz \dxkm1 \right]_i \left[\sszk
        \dsk + \ssz \dskm1 \right]_i - (1 + \gamma_1 \ssz) \sszk^2
      [\dxk]_i [\dsk]_i \\
      & \ge (1 - \gamma) \ssz \sigma_{k + 1} (1 - \sszk (1 - \sigma_k)) \mu_k 
        - (1 + \gamma_1 \ssz) \sszk^2 
        \left( \frac{1 \! + \! \gamma}{\gamma} \right) 
        \frac{n}{2^{3/2}} \mu_k \\
      & \quad - \frac{n}{2^{3/2} \gamma} 
        \left\{ \left[ \ssz + \sszk (1 \! - \! \ssz) 
                       + \frac{|1 \! + \! \gamma_1| \ssz \sszk^2}{2^{3/2}} \right]^2 
        \left( \frac{1 \! + \! \gamma}{\gamma} \right)^2 n 
        + [1 - (1 - \ssz (1 \! - \! \sigma_{k + 1})) 
          (1 \! - \! \sszk (1 \! - \! \sigma_k))]^2 \right\} \mu_k \\
      & = \mu_k \Bigg\{ (1 - \gamma) \ssz \sigma_{k + 1} (1 - \sszk (1 - \sigma_k)) 
        - (1 + \gamma_1 \ssz) \sszk^2 \left( \frac{1 + \gamma}{\gamma} \right) \frac{n}{2^{3/2}} \\
      & \quad - \left[ \ssz + \sszk (1 \! - \! \ssz) + 
        \frac{|1 + \gamma_1| \ssz \sszk^2}{2^{3/2}} \right]^2 
        \left( \frac{1 \! + \! \gamma}{\gamma} \right)^2 \frac{n^2}{2^{3/2} \gamma} 
        - [1 \! - \! (1 \! - \! \ssz (1 \! - \! \sigma_{k + 1})) 
         (1 \! - \! \sszk (1 \! - \! \sigma_k))]^2 \frac{n}{2^{3/2} \gamma} \Bigg\}.
    \end{split}
  \end{equation}
  We will rearrange this expression and represent it in a form 
  $K_1 \bar \alpha - K_2 {\bar \alpha}^2$ with $K_1, K_2 > 0$. 
  Next, we will show that ~\eqref{t:nsconv:e1} is non-negative 
  for sufficiently small values of $\sszk$ and $\ssz$. 
  To deliver the desired results every term inside the curly 
  brackets in~\eqref{t:nsconv:e1} will be bounded. 
  Since $\sszk \le \frac{1}{2} \ssz \le \frac{1}{2} $, we have 
  $1 - \sszk (1 - \sigma_k) \ge 1 - \frac{1}{2} (1 - \sigma_k) \ge \frac{1}{2}$ 
  hence for the first term in~\eqref{t:nsconv:e1} we obtain 
  \begin{equation}
    \label{t:nsconv:e2}
    (1 - \gamma) \ssz \sigma_{k + 1} (1 - \sszk (1 - \sigma_k)) 
    \ge \frac{1}{2} [(1 - \gamma) \sigma_{k + 1}] \ssz.
  \end{equation}
  Using condition $\ssz \ge 2 \sszk$ (guaranteed by~\eqref{t:nsconv:bnds}),
  $n \ge 1$ and Lemma~\ref{l:nsgbnd}, the second term becomes
  \begin{equation} \label{t:nsconv:e4}
    \begin{split}
      (1 + \gamma_1 \ssz) \sszk^2 \left( \frac{1 + \gamma}{\gamma}
      \right) \frac{n}{2^{3/2}} & \le (\sszk^2 + |\gamma_1| \ssz
      \sszk^2) \left( \frac{1 + \gamma}{\gamma} \right)
      \frac{n}{2^{3/2}} \le \left(\sszk^2 + \frac{2 \sqrt{n}}{\gamma
          (1 - \sigma_k)} \ssz \sszk \right) \left( \frac{1 +
          \gamma}{\gamma} \right) \frac{n}{2^{3/2}} \\
      & \le \left(\frac{\ssz^2}{4} + \frac{\sqrt{n}}{\gamma (1 -
          \sigma_k)} \ssz^2 \right) \left( \frac{1 + \gamma}{\gamma}
      \right) \frac{n}{2^{3/2}} = \left(\frac{1}{4} +
        \frac{\sqrt{n}}{\gamma (1 - \sigma_k)} \right) \left( \frac{1
          + \gamma}{\gamma} \right) \frac{n}{2^{3/2}} \ssz^2 \\
      & \le \left(2 +
        \frac{1}{\gamma (1 - \sigma_k)} \right) \left( \frac{1
          + \gamma}{\gamma} \right) \frac{n \sqrt{n} }{2^{3/2} \gamma} \ssz^2.
    \end{split}
  \end{equation}
  By applying condition $\ssz \ge 2 \sszk$ again, we conclude that
  \begin{equation}
    \label{t:nsconv:e3}
    \ssz + \sszk (1 - \ssz) = \ssz + \sszk - \ssz \sszk 
    \le \ssz + \frac{1}{2} \ssz - \ssz \sszk 
    \le \frac{3}{2} \ssz
  \end{equation}
  and, using~\eqref{t:nsconv:e3} and the same arguments as before, 
  for the third term in \eqref{t:nsconv:e1} we obtain
  \begin{equation}
    \label{t:nsconv:e5}
    \begin{split}
      \left[ \ssz + \sszk (1 - \ssz) + \frac{|1 + \gamma_1| \ssz
          \sszk^2}{2^{3/2}} \right]^2 \left( \frac{1 + \gamma}{\gamma}
      \right)^2 \frac{n^2}{2^{3/2} \gamma} & \le \left[ \frac{3}{2}
        \ssz + \left(1 + \frac{2 \sqrt{n}}{\gamma (1 - \sigma_k)
            \sszk} \right) \frac{\ssz \sszk^2}{2^{3/2}} \right]^2
      \left( \frac{1 + \gamma}{\gamma} \right)^2 \frac{n^2}{2^{3/2}
        \gamma} \\
      & \le \left( \frac{3}{2} + \frac{1}{2} + \frac{\sqrt{n}}{\gamma
          (1 - \sigma_k)} \right)^2 \left( \frac{1 +
          \gamma}{\gamma} \right)^2 \frac{n^2}{2^{3/2} \gamma} \ssz^2 \\
      & \le \left( 2 + \frac{1}{\gamma (1 - \sigma_k)} \right)^2
      \left( \frac{1 + \gamma}{\gamma} \right)^2 \frac{n^3}{2^{3/2}
        \gamma} \ssz^2.
    \end{split}
  \end{equation}
  In a similar fashion, the last term of~\eqref{t:nsconv:e1} can be
  bounded as
  \begin{equation}
    \label{t:nsconv:e6}
    \begin{split}
      [1 \! - \! (1 \! - \! \ssz (1 \! - \! \sigma_{k + 1})) 
      & (1 - \sszk (1 - \sigma_k))]^2 \frac{n}{2^{3/2} \gamma} 
        = [1 \! - \! (1 \! - \! \sszk (1 \! - \! \sigma_k)) + \ssz (1 \! - \! \sigma_{k + 1})
          (1 \! - \! \sszk (1 \! - \! \sigma_k)) ]^2 \frac{n}{2^{3/2} \gamma} \\
      & = [\sszk (1 - \sigma_k) + \ssz (1 - \sigma_{k + 1}) (1 -
      \sszk (1 - \sigma_k)) ]^2 \frac{n}{2^{3/2} \gamma} \\
      & \le \left[\frac{(1 - \sigma_k)}{2} 
            + (1 - \sigma_{k + 1}) (1 - \sszk (1 - \sigma_k)) \right]^2 
              \frac{n}{2^{3/2} \gamma} \ssz^2 
        \quad \le \quad 4 \frac{n}{2^{3/2} \gamma} \ssz^2.
    \end{split}
  \end{equation}
  Since $n \ge 1$ and assuming without loss of generality that
  $\left( 2 + \frac{1}{\gamma (1 - \sigma_k)} \right) \left( \frac{1 +
      \gamma}{\gamma} \right) \ge 4$, using~\eqref{t:nsconv:e2},
  \eqref{t:nsconv:e4}, \eqref{t:nsconv:e5} and~\eqref{t:nsconv:e6}
  in~\eqref{t:nsconv:e1} we have that
  \begin{equation*}
    \begin{split}
      \left[ x^{k + 1} + \bar \alpha \dxkm1 \right]_i 
      & \left[ z^{k + 1} + \bar \alpha \dskm1 \right]_i - \gamma \mu(\ssz) \ge \\
      & \ge \Bigg\{ \frac{1}{2} [(1 - \gamma) \sigma_{k + 1}] \ssz 
            - 3 \left( 2 + \frac{1}{\gamma (1 - \sigma_k)} \right)^2 
            \left( \frac{1 + \gamma}{\gamma} \right)^2 \frac{n^3}{2^{3/2} \gamma} \ssz^2 \Bigg\} \mu_k \\
      & \ge \Bigg\{ \frac{1}{2} [(1 - \gamma) \sigma_{min}] \ssz
            - 3 \left( 2 + \frac{1}{\gamma (1 - \sigma_{max})} \right)^2 
            \left( \frac{1 + \gamma}{\gamma} \right)^2 \frac{n^3}{2^{3/2} \gamma} \ssz^2 \Bigg\} \mu_k.
    \end{split}
  \end{equation*}
  Therefore, in order to guarantee that 
  $\left[ x^{k + 1} + \bar \alpha \dxkm1 \right]_i \left[ z^{k + 1} +
    \bar \alpha \dskm1 \right]_i \ge \gamma \mu(\ssz)$ 
  for $i = 1, \dots, n$, it is sufficient that the quasi-Newton step-size
  $\ssz$ satisfies 
  \begin{equation}
    \label{t:nsconv:e7}
    \ssz \le \frac{\frac{1}{2} (1 - \gamma) \sigma_{min}}
             {\frac{3}{2^{3/2} \gamma} 
                   \left( 2 + \frac{1}{\gamma (1 - \sigma_{max})} \right)^2 
                   \left( \frac{1 + \gamma}{\gamma} \right)^2 n^3} 
          = \frac{(1 - \gamma) l}{n^3}.
  \end{equation}
  We now set $\xi = 1/\gamma$ in~\eqref{ns:e1}. In order to show that
  the resulting expression is non-positive, we use the same arguments
  as before, that is, Lemma~\ref{l:ns:qnstepbnd}   
  and equations~\eqref{t:nsconv:e2}, \eqref{t:nsconv:e4}, \eqref{t:nsconv:e5} 
  and~\eqref{t:nsconv:e6} to obtain
  \begin{equation}
    \label{t:nsconv:e8}
    \begin{split}
      \left[ x^{k + 1} \right. & + \left. \bar \alpha \dxkm1 \right]_i
      \left[ z^{k + 1}
        + \bar \alpha \dskm1 \right]_i - (1/\gamma) \mu(\ssz) = \\
      & = (1 - \ssz) (x_i^{k + 1} z_i^{k + 1} - (1/\gamma) \mu_{k +
        1}) + (1 - (1/\gamma)) \ssz \sigma_{k + 1} \mu_{k + 1}\\
      & \quad + \left[\sszk \dxk + \ssz \dxkm1 \right]_i \left[\sszk
        \dsk + \ssz \dskm1 \right]_i - (1 + \gamma_1 \ssz) \sszk^2
      [\dxk]_i [\dsk]_i \\
      & \le (1 - (1/\gamma)) \ssz \sigma_{k + 1} \mu_{k + 1} +
      \left[\sszk \dxk + \ssz \dxkm1 \right]_i \left[\sszk \dsk + \ssz
        \dskm1 \right]_i + |1 + \gamma_1 \ssz| \sszk^2
      | [\dxk]_i [\dsk]_i | \\
      & \le (1 - (1/\gamma)) \ssz \sigma_{k + 1} (1 - \sszk (1 -
      \sigma_k)) \mu_k + |1 + \gamma_1 \ssz| \sszk^2 \left(
        \frac{1 + \gamma}{\gamma} \right) \frac{n}{2^{3/2}} \mu_k \\
      & \quad + \frac{n}{2^{3/2} \gamma} 
        \left\{ \left[ \ssz \! + \! \sszk (1 \! - \! \ssz) 
                \! + \! \frac{|1 \! + \! \gamma_1| \ssz \sszk^2}{2^{3/2}} \right]^2 
        \left( \frac{1 \! + \! \gamma}{\gamma} \right)^2 n 
        + [1 \! - \! (1 \! - \! \ssz (1 \! - \! \sigma_{k + 1})) 
                     (1 \! - \! \sszk (1 \! - \! \sigma_k))]^2 \right\} \mu_k \\
      & = \mu_k \Bigg\{ |1 + \gamma_1 \ssz| \sszk^2 \left( \frac{1 +
          \gamma}{\gamma} \right) \frac{n}{2^{3/2}} + \left[ \ssz +
        \sszk (1 - \ssz) + \frac{|1 + \gamma_1| \ssz \sszk^2}{2^{3/2}}
      \right]^2 \left( \frac{1 +
          \gamma}{\gamma} \right)^2 \frac{n^2}{2^{3/2} \gamma} \\
      & \hspace{1.2cm} + [1 - (1 - \ssz (1 - \sigma_{k + 1})) (1 -
      \sszk (1 - \sigma_k))]^2 \frac{n}{2^{3/2} \gamma} - \left( \frac{1 -
        \gamma}{\gamma} \right) \ssz \sigma_{k + 1} (1 - \sszk (1 - \sigma_k))
      \Bigg\} \\
      & \le \mu_k \left\{ 3 \left( 2 + \frac{1}{\gamma (1 - \sigma_k)}
        \right)^2 \left( \frac{1 + \gamma}{\gamma} \right)^2
        \frac{n^3}{2^{3/2} \gamma} \ssz^2 - \frac{1}{2} \sigma_{k + 1} 
            \left( \frac{1 - \gamma}{\gamma} \right) \ssz
      \right\} \\
      & \le \mu_k \left\{ 3 \left( 2 + \frac{1}{\gamma (1 - \sigma_{max})}
        \right)^2 \left( \frac{1 + \gamma}{\gamma} \right)^2
        \frac{n^3}{2^{3/2} \gamma} \ssz^2 - \frac{1}{2} \sigma_{min} 
            \left( \frac{1 - \gamma}{\gamma} \right) \ssz
      \right\}.
    \end{split}
  \end{equation}
  Therefore, if
  \begin{equation}
    \label{t:nsconv:e9}
    \ssz \le \frac{\left( \dfrac{1 - \gamma}{\gamma} \right) l}{n^3},
  \end{equation}
  then by~\eqref{t:nsconv:e8} we have that
  $\left[ x^{k + 1} + \bar \alpha \dxkm1 \right]_i 
   \left[ z^{k + 1} + \bar \alpha \dskm1 \right]_i 
   \le (1/\gamma) \mu(\ssz)$ for all $i = 1, \dots, n$. 
  Observe that bound~\eqref{t:nsconv:e7} is tighter than ~\eqref{t:nsconv:e9}
  because $\gamma \in (0,1)$ hence~\eqref{t:nsconv:e7} appears 
  as an upper bound on $\bar \alpha$ in~\eqref{t:nsconv:bnds}. 
  %
  %
  The remaining bounds in~\eqref{t:nsconv:bnds} are consistent with 
  the need to satisfy $\ssz \geq 2 \sszk$.  
  %

    %
    It remains to show that $x^{k + 2}$ and $z^{k + 2}$ are strictly
    positive. Similarly to Theorem~\ref{t:n2conv}, if, by
    contradiction, $x_i^{k + 2} \le 0$ or $z_i^{k + 2} \le 0$ for some
    $i$, then we must have $x_i^{k + 2} < 0$ and $z_i^{k + 2} < 0$. On
    one hand, we already know that $x^k_i z^k_i \ge \gamma \mu_k$, by
    Assumption~\ref{a:ns}. Hence, by
    Lemma~\ref{l:ns:qnstepbnd}, \eqref{t:nsconv:e5},
    \eqref{t:nsconv:e6} and similar arguments to those used in this proof
    \begin{equation*}
      \begin{split}
        x^k_i z^k_i & <
        \left[ \sszk \dxk + \ssz \dxkm1
        \right]_i \left[ \sszk \dsk + \ssz \dskm1 \right]_i
        \le \| \left( \bar \alpha_k \dXk + \bar
          \alpha \dXkm1 \right) \left( \bar \alpha_k \dSk + \bar
          \alpha \dSkm1 \right) e \| \\
        & \le \frac{3}{2^{3/2} \gamma} 
              \left(2 + \frac{1}{\gamma (1 - \sigma_{max})} \right)^2 
              \left( \frac{1 + \gamma}{\gamma} \right)^2 \ssz^2 \mu_k n^3 
          = \frac{\sigma_{min} \ssz^2 \mu_k n^3}{2 l}.
      \end{split}
    \end{equation*}
    Hence we conclude that $\gamma < \sigma_{min} n^3 \ssz^2 / (2 l)$
    which, together with condition $\gamma \ge \sigma_{min} / 4$
    and~\eqref{t:nsconv:bnds}, implies the following absurd
    \[
      \frac{\sigma_{min}}{4} \le \gamma < \frac{\sigma_{min} n^3
        \ssz^2}{2 l} \le \frac{\sigma_{min} n^3 \ssz}{2 l} \le
      \frac{\sigma_{min} (1 - \gamma)}{4} \le \frac{\sigma_{min}}{4}.
    \]
    By Assumption~\ref{a:pdfeas}, we have that $\xysk{k + 2}$ belongs
    to $\F$, hence we conclude that it belongs to $\nN_s(\gamma)$.
\end{proof}

To deliver the worst-case polynomial complexity, we follow the same 
arguments as those presented after the proof of Theorem~\ref{t:n2conv} 
and consider only Newton iterations which are simpler to analyze. 
By~\eqref{t:nsconv:bnds}, we have that
\[
  \sszk \ge \frac{(1 - \gamma) l}{2 n^3}
\]
and, by~\cite[Lemma 5.1]{Wright1997},
\[
  \mu_{k + 1} = (1 - \sszk (1 - \sigma_k)) \mu_k \le \left(1 -
    \frac{(1 - \gamma) l}{2 n^3} \right) \mu_k, \quad k = 0, 2, 4,
  \dots,
\]
from which the convergence with the worst-case iteration 
complexity of $O(n^3)$ can be established easily.

  Because of the close relation of the symmetric
  neighborhood and the $\nN_{-\infty}(\gamma)$ neighborhood, one
  should expect similar complexity results to that of
  Theorem~\ref{t:nsconv}. However, if the $\nN_s(\gamma)$ neighborhood
  is not applied in Lemma~\ref{l:nsgbnd}, then a naive approach is to
  bound $\sum_{i = 1}^n \left( x^{k + 1}_i z^{k + 1}_i \right)^2$ by
  $\mu_{k + 1}^2 n^2$. Such approach would increase the worst-case
  iteration complexity to $O(n^4)$, as the degree would be increased
  in equation~\eqref{t:nsconv:e5}. A different approach to reduce the
  worst-case polynomial degree when working with the
  $\nN_{-\infty}(\gamma)$ neighborhood is subject to future
  developments.


%
%
%

\section{Worst-case complexity in the infeasible case} 
\label{infeas}

In this section, we analyze worst-case iteration complexity for linear
programming problems without assuming feasibility of the starting point. 
The analysis uses some ideas developed in Section~\ref{feasible} and 
follows~\cite[Chapter 6]{Wright1997}, by adding extra requirements 
on the computation of step-size $\sszk$. The proofs are modified 
to consider the symmetric neighborhood and admit the quasi-Newton steps.

We start by defining $r_b^k$ and $r_c^k$ to be the primal and dual 
infeasibility vectors at point $\xysk{k}$ which appear in the right-hand 
side of equation~\eqref{def:newtonsys}: 
\[
  r_b^k = A x^k - b \quad \text{and} \quad r_c^k = A^T \lambda^k + z^k - c.
\]
Then, given an initial guess $\xysk{0}$ and parameters $\beta \ge 1$
and $\gamma \in (0, 1)$, the infeasible version of the symmetric
neighborhood is defined as follows
\[
  \nN_s(\gamma, \beta) \doteq \left\{ (x, \lambda, z)\ \Big|\ \ \|(r_b,
    r_c)\| \le \frac{\|(r_b^0, r_c^0)\|}{\mu_0} \beta \mu, x > 0, z >
    0\ \text{and}\ \gamma \mu \le x_i z_i \le \frac{1}{\gamma} \mu, i
    = 1,\dots, n \right\}.
\]
In order to obtain complexity results, $\sszk$ needs to be computed in
such a way that the new iterate remains in $\nN_s(\gamma, \beta)$ and
a sufficient decrease condition for $\mu_k$ is satisfied. More
precisely, given $\alphadec \in (0, 1)$, $\sszk$ is the largest value
in $[0, 1]$ (or a fixed fraction of it) such that
\begin{equation} \label{t:inf:sufdec} %
  \xysk{k} + \sszk \dxysk
  \in \nN_s(\gamma, \beta) \quad \text{and} \quad \mu(\sszk) \le (1 -
  \alphadec \sszk) \mu_k.
\end{equation}
The goal of this section is to show that, assuming that one
quasi-Newton step is performed from a point in $\nN_s(\gamma, \beta)$
by Algorithm~\ref{alg:pdqnipm}, then conditions~\eqref{t:inf:sufdec}
are satisfied by the new point. When only Newton steps are taken, it
is shown in~\cite[Lemma~6.7]{Wright1997} that there is an interval
$[0, \hat \alpha]$ such that~\eqref{t:inf:sufdec} holds for the
$\nN_{-\infty}(\gamma, \beta)$ neighborhood. This is the case if 
a special starting point is used and $\hat \alpha \ge \bar \delta / n^2$,
where $\bar \delta$ is a constant independent of $n$. In
Lemma~\ref{l:inf:symmetric}, we extend those results in order to
ensure that the iterates belong to $\nN_s(\gamma, \beta)$ when only
Newton steps are taken.

First, let us recall some results from~\cite{Wright1997} 
that will be used frequently. We define the scalar
$\displaystyle \nu_k = \prod_{i = 0}^{k - 1} (1 - \ssz_i)$, 
which allows us to write vectors $r_b^k$ and $r_c^k$ as
$r_b^k = \nu_k r_b^0$ and $r_c^k = \nu_k r_c^0$, respectively. 
The parameter $\sigma_k$ used in~\eqref{def:qnsys} satisfies 
$0 < \sigma_{min} \le \sigma_k \le \sigma_{max} < 1$ for all
$k$. We also assume that a special initial point is used 
in Algorithm~\ref{alg:pdqnipm}, given by
\begin{equation} 
  \label{t:infinit} %
  \xysk0 = \begin{bmatrix} \xi e & 0 & \xi e \end{bmatrix}^T,
\end{equation}
where $\xi$ is such that $\| (x^*, z^*) \|_\infty \le \xi$, 
for some primal-dual solution $\xysk*$.

Let $k$ be a Newton iteration, $\xysk{k} \in \nN_s(\gamma, \beta)$ 
and $D^k = (X^k)^{1/2} (Z^k)^{-1/2}$. Also, let $\omega = 9 \beta / \gamma^{1/2}$ 
be a constant independent of $n$. When~\eqref{t:infinit} is used 
as the starting point, the following bounds hold:
\begin{align}
  \nu_k \| (x^k, z^k) \|_1 & \le \frac{4 \beta}{\xi} n \mu_k \label{t:inf:xz1bnd} \\
  \mu_0 & = \xi^2 \label{t:inf:mu0} \\
  \|(D^k)^{-1} \dxk \| & \le \omega n \mu_k^{1/2} 
  \quad \text{and} \quad 
  \|D^k \dsk \|  \le \omega n \mu_k^{1/2}. \label{t:inf:dxdzbnd}
\end{align}
The proofs for these bounds can be found in Lemmas~6.4 and~6.6
of~\cite{Wright1997}, respectively.

We start the analysis from looking at Newton step in $\nN_s(\gamma, \beta)$ 
neighborhood. 
\begin{lemma}
  \label{l:inf:symmetric}
  If $\xysk{k} \in \nN_s(\gamma, \beta)$ and the Newton step is taken
  at iteration $k$, then there exists a constant $\bar {\bar \delta}$ 
  independent of $n$ and a value $\hat \alpha \ge \bar {\bar \delta} / n^2$, 
  independent of $k$, such that for all $\sszk \in [0, \hat \alpha]$
  \begin{align}
    (1 - \sszk) {x^k}^T z^k 
    & \le (x^k + \sszk \dxk)^T(z^k + \sszk \dsk) 
      \le (1 - \alphadec \sszk) {x^k}^T z^k \label{l:inf:sym:e0} \\ 
    \gamma \mu(\sszk)
    & \le [x^k + \sszk \dxk]_i [z^k + \sszk \dsk]_i 
      \le \frac{1}{\gamma} \mu(\sszk) \label{l:inf:sym:e1}
  \end{align}
\end{lemma}

\begin{proof}
  We only need to address the right inequality
  of~\eqref{l:inf:sym:e1}, since all the other results were detailed
  in~\cite[Lemma~6.7]{Wright1997}. Since $k$ is a Newton iteration,
  by~\eqref{t:inf:dxdzbnd}, we obtain 
  \[
    \lvert \dxk_i \dsk_i \rvert \le \lvert [(D^k)^{-1}]_{ii} \dxk_i
    \rvert \lvert [D^k]_{ii} \dsk_i \rvert \le \| (D^k)^{-1} \dxk \|
    \| D^k \dsk \| \le \omega^2 n^2 \mu_k.
  \]
  As $\dxysk$ solves~\eqref{def:newtonsys} and
  $\xysk{k} \in \nN_s(\gamma, \beta)$, using a component-wise version
  of~\eqref{CProds-kp1} we get
  \begin{equation}
    \label{l:inf:sym:e2}
    \begin{split}
      [x^k + \sszk \dxk]_i [z^k + \sszk \dsk]_i 
      & = (1 - \sszk) x^k_i z^k_i + \sszk \sigma_k \mu_k + \sszk^2 \dxk_i \dsk_i \\
      & \le \frac{1 - \sszk}{\gamma} \mu_k + \sszk \sigma_k \mu_k +
      \sszk^2 \omega^2 n^2 \mu_k.
    \end{split}
  \end{equation}
  Using similar arguments and equation~\eqref{CProds-kp1} again, we also obtain
  \begin{equation}
    \label{l:inf:sym:e3}
    \begin{split}
      \frac{1}{\gamma} \mu(\sszk) 
        =   \frac{1}{\gamma} \frac{(x^k + \sszk \dxk)^T(z^k + \sszk \dsk)}{n} 
      & \ge \frac{1 - \sszk}{\gamma} \mu_k + \frac{\sszk \sigma_k}{\gamma} \mu_k 
            - \frac{\sszk^2 \lvert {\dxk}^T \dsk \rvert}{\gamma n} \\
      & \ge \frac{1 - \sszk}{\gamma} \mu_k + \frac{\sszk \sigma_k}{\gamma} \mu_k 
            - \frac{\sszk^2 \omega^2 n}{\gamma} \mu_k.
    \end{split}
  \end{equation}
  Using~\eqref{l:inf:sym:e2}, \eqref{l:inf:sym:e3} and the fact that
  $n \ge 1$, we obtain
  \begin{equation*}
    \begin{split}
      [x^k + \sszk \dxk]_i [z^k + \sszk \dsk]_i - \frac{1}{\gamma}
      \mu(\sszk) & \le \sszk \sigma_k \left(1 - \frac{1}{\gamma}
      \right) \mu_k + \sszk^2 \omega^2 n^2 \left( 1 + \frac{1}{\gamma}
      \right) \mu_k.
    \end{split}
  \end{equation*}
  The right-hand side of this inequality is non-positive if
  $\sszk \le \dfrac{\sigma_{min}(1 - \gamma)}{\omega^2 n^2 (1 +
    \gamma)}$. By defining $\bar {\bar \delta}$ as the minimum of
  $\dfrac{\sigma_{min}(1 - \gamma)}{\omega^2 (1 + \gamma)}$ 
  and $\bar \delta$ defined in~\cite[Lemma~6.7]{Wright1997}, 
  we obtain the desired result.
\end{proof}

By Lemma~\ref{l:inf:symmetric}, if $k$ is a Newton
iteration of Algorithm~\ref{alg:pdqnipm} and
$\sszk \le \bar {\bar \delta} / n^2$, then point $\xysk{k + 1}$
satisfies~\eqref{t:inf:sufdec}. If only Newton steps are made, 
then $O(n^2)$ worst-case iteration complexity can be proved 
for $\nN_s(\gamma, \beta)$ (see~\cite[Chapter 6]{Wright1997}).

Based on the previous paragraphs, we set some assumptions that will be
used in the remaining of this section. As usual, at iteration $k$ 
the Newton step is taken and at iteration $k + 1$ the quasi-Newton 
step is made.  
We assume that $\xysk{0}$ satisfies~\eqref{t:infinit},
both $\xysk{k}$ and $\sszk$ satisfy~\eqref{t:inf:sufdec}, and
$\xysk{k + 1} \in \nN_s(\gamma, \beta)$. In our analysis, we will use
extensively various well known results for Newton steps, such as
properties~\eqref{t:inf:xz1bnd}--\eqref{l:inf:sym:e1}.
We observe that, while in the standard Newton approach one has 
to compute bounds for ${\dxk}^T \dsk$, in the quasi-Newton approach 
(see for example~\eqref{CP-afterQN}), 
it is necessary to additionally get bounds for $\gamma_1$ and
$(\sszk \dxk + \ssz \dxkm1)^T (\sszk \dsk + \ssz \dskm1)$. 
In the next lemma, we give a bound for $\gamma_1$ without assuming 
feasibility of iterates.



\begin{lemma}
  \label{l:inf:gamma1bnd}
  Let $k + 1$ be a quasi-Newton iteration of
  Algorithm~\ref{alg:pdqnipm}. Suppose that $v$ in
  Lemma~\ref{l:bbroyd} is given by the right-hand side
  of~\eqref{def:newtonsys} at iteration $k + 1$ and
  $\sszk \in (0, 1]$.
  %
  %
  %
  Then, there exists a constant $C_3 \ge 1$, independent
  of $n$ and $k$ such that
  \begin{equation} \label{t:inf:gamma1bnd} %
    | \gamma_1 | \le C_3 \frac{\sqrt{n}}{\sszk}.
  \end{equation}
\end{lemma}

\begin{proof}
  Using scalar $\nu_{k + 1}$ defined in the beginning of this section,
  vector $v$ at iteration $k + 1$ is given by
  \begin{equation} \label{l:infvbnd:e1} %
    v =
    \begin{bmatrix}
      0 \\ 0 \\ \sigma_{k + 1} \mu_{k + 1} e
    \end{bmatrix}
    - F(x^{k + 1}, \lambda^{k + 1}, z^{k + 1}) =
    \begin{bmatrix}
      - r_c^{k + 1} \\ 
      - r_b^{k + 1} \\ 
      \sigma_{k + 1} \mu_{k + 1} e - X^{k + 1} Z^{k + 1} e
    \end{bmatrix}
    =
    \begin{bmatrix}
      - \nu_{k + 1} r_c^0 \\ 
      - \nu_{k + 1} r_b^0 \\ 
      \sigma_{k + 1} \mu_{k + 1} e - X^{k + 1} Z^{k + 1} e
    \end{bmatrix} .
  \end{equation}
  We observe that, since
  $e^T (\mu_{k + 1} e - X^{k + 1} Z^{k + 1} e) = 0$, 
  by~\eqref{t:inf:sufdec} we obtain 
  \begin{equation} \label{l:infvbnd:e2} %
    \begin{split}
      \| \sigma_{k + 1} \mu_{k + 1} e - X^{k + 1} Z^{k + 1} e \|^2 
      & = \| (\sigma_{k + 1} - 1) \mu_{k + 1} e + (\mu_{k + 1} e - X^{k + 1} Z^{k + 1} e) \|^2 \\
      & = (\sigma_{k + 1} - 1)^2 \mu_{k + 1}^2 n + \| \mu_{k + 1} e - X^{k + 1} Z^{k + 1} e \|^2 \\
      & = (\sigma_{k + 1} - 1)^2 \mu_{k + 1}^2 n + \| X^{k + 1} Z^{k + 1} e \|^2 - n \mu_{k + 1}^2 \\
      & \le - (2 - \sigma_{k + 1}) \sigma_{k + 1} n \mu_{k + 1}^2 
      + \gamma^{-2} n \mu_{k + 1}^2\ \le\ \gamma^{-2} n \mu_{k + 1}^2 .
      %
      %
      %
      %
    \end{split}
  \end{equation}
  Taking the 2-norm in~\eqref{l:infvbnd:e1} and
  using~\eqref{t:inf:mu0} and~\eqref{l:infvbnd:e2}, we get 
  %
  \begin{equation} \label{l:infvbnd:e3} %
    \begin{split}
      \| v \|^2
      & = \|(r_c^{k + 1}, r_b^{k + 1})\|^2 + \| \sigma_{k +
        1} \mu_{k + 1} e - X^{k + 1} Z^{k + 1} e \|^2 \le \left(
        \frac{\|(r_c^0, r_b^0)\| \beta}{\mu_0} \right)^2 \mu^2_{k + 1}
      + \gamma^{-2} n \mu_{k + 1}^2 \\
      & \le \left[ \left( \frac{\|(r_c^0, r_b^0)\| \beta}{\xi^2}
        \right)^2 + \gamma^{-2} \right] n \mu^2_{k + 1}.
    \end{split}
  \end{equation}
  By defining $\rho_k = \alphadec > 0$, we can see that
  condition~\eqref{t:inf:sufdec} results in the sufficient decrease
  condition of Lemma~\ref{l:ykbnd}. Since $\sszk \in (0, 1]$, this
  ensures that $\| y_k \| > 0$. Using Lemma~\ref{l:ykbnd},
  \eqref{t:inf:sufdec} and~\eqref{l:infvbnd:e3} in
  equation~\eqref{projection}, we conclude that
  \begin{equation*}
    \begin{split}
      | \gamma_1 | \le \frac{\| v \|}{\| y_k \|} 
      & \le \frac{\left[
          \left( \dfrac{\|(r_c^0, r_b^0)\| \beta}{\xi^2} \right)^2 +
          \gamma^{-2} \right]^{1/2} \sqrt{n} \mu_{k + 1}}{\frac{\alphadec \sszk}{2} \mu_k}
      \le C_3 \frac{\sqrt{n}}{\sszk},
    \end{split}
  \end{equation*}
  with
  $C_3 = 2 (\gamma \alphadec)^{-1} 
  \left[ \left( \dfrac{\|(r_c^0, r_b^0)\| \beta \gamma}{\xi^2} \right)^2 
    \! + \! 1 \right]^{1/2}$.
  In the last inequality we used~\eqref{t:inf:sufdec} to yield 
  $\mu_{k + 1} / \mu_k \le 1$.
\end{proof}

Our goal now is to bound the term
$(\sszk \dxk + \ssz \dxkm1)^T (\sszk \dsk + \ssz \dskm1)$
in~\eqref{CP-afterQN}. We follow the same approach as~\cite{Wright1997} 
but compute, instead, bounds of $(D^k)^{-1} (\sszk \dxk + \ssz \dxkm1)$ 
and $D^k (\sszk \dsk + \ssz \dskm1)$, where $D^k = (X^k)^{1/2} (Z^k)^{-1/2}$. 
It is important to note that matrix $D^k$ is related to the Newton 
iteration, hence we can use Lemma~\ref{l:bbroyd} and the enjoyable 
properties of the true Jacobian.

First, we show what happens when matrix $A$ multiplies the combined
direction $\sszk \dxk + \ssz \dxkm1$:
\begin{equation*}
  \begin{split}
    A (\sszk \dxk & + \ssz \dxkm1) = \sszk A \dxk + \ssz A \dxkm1 = -
    \sszk (A x^k - b) - \ssz (A x^{k + 1} - b) \\
    & = - \sszk (A x^k - b) - \ssz (A (x^k + \sszk \dxk) - b)
    \ = \ - \sszk (A x^k - b) - \ssz (A x^k - b) + \sszk \ssz (A x^k - b) \\
    & = (\sszk \ssz - \sszk - \ssz) (A x^k - b)
    \ = \ (\sszk \ssz - \sszk - \ssz) \nu_k (A x^0 - b) \\
    & = A \left[ (\sszk \ssz - \sszk - \ssz) \nu_k (x^0 - x^*) \right],
  \end{split}  
\end{equation*}
where $x^*$ is the primal solution of~\eqref{def:qp} used
to define constant $\xi$ in~\eqref{t:infinit}. By defining
$\hat \nu_k = - (\sszk \ssz - \sszk - \ssz) \nu_k 
            = (1 - (1 - \sszk) (1 - \ssz)) \nu_k$, 
we can see that the point
$\hat x = \sszk \dxk + \ssz \dxkm1 + \hat \nu_k (x^0 - x^*)$ is such
that $A \hat x = 0$. Similar arguments can be used to show that
$\hat z = \sszk \dsk + \ssz \dskm1 + \hat \nu_k (z^0 - z^*)$ and
$\hat \lambda = \sszk \dyk + \ssz \dykm1 + \hat \nu_k (\lambda^0 -
\lambda^*)$ satisfy $A^T \hat \lambda + \hat z = 0$. Therefore, it is
not hard to see that
\begin{equation} \label{t:inf:scalarp} %
  {\hat x}^T \hat z = \left( \sszk \dxk + \ssz \dxkm1 + \hat \nu_k
    (x^0 - x^*) \right)^T \left( \sszk \dsk + \ssz \dskm1 + \hat \nu_k
    (z^0 - z^*) \right) = 0.
\end{equation}

We now multiply the third row of the coefficient matrix
in~\eqref{def:newtonsys} by
$\begin{bmatrix} \hat x & \hat \lambda & \hat z \end{bmatrix}$ to
obtain
\begin{equation*}
  \begin{split}
    Z^k \hat x + X^k \hat z & = Z^k \left( \sszk \dxk + \ssz \dxkm1 +
      \hat \nu_k (x^0 - x^*) \right) + X^k \left( \sszk \dsk + \ssz
      \dskm1 + \hat \nu_k (z^0 - z^*) \right)\\
    & = Z^k ( \sszk \dxk + \ssz \dxkm1) + X^k ( \sszk \dsk + \ssz
    \dskm1 ) + \hat \nu_k Z^k (x^0 - x^*) + \hat \nu_k X^k ( z^0 -
    z^*).
  \end{split}
\end{equation*}
Multiplying this equation on both sides by $(X^k Z^k)^{-1/2}$, we get 
\begin{equation} \label{t:inf:hateq} %
  \begin{split}
    (D^k)^{-1} \hat x + D^k \hat z & = (X^k Z^k)^{-1/2} \left[ Z^k (
      \sszk \dxk + \ssz \dxkm1) + X^k ( \sszk \dsk + \ssz
      \dskm1 ) \right] \\
    & \quad + \hat \nu_k (D^k)^{-1} (x^0 - x^*) + \hat \nu_k {D^k} (
    z^0 - z^*).
  \end{split}
\end{equation}
Using~\eqref{t:inf:scalarp} and~\eqref{t:inf:hateq}, we conclude that
\begin{equation*}
  \begin{split}
    \| (D^k)^{-1} \hat x \|^2 + \| D^k \hat z \|^2 
    & = \| (D^k)^{-1} \hat x + D^k \hat z \|^2 \\
    & = \Big\| (X^k Z^k)^{-1/2} 
           \left[ Z^k ( \sszk \dxk + \ssz \dxkm1) 
                + X^k ( \sszk \dsk + \ssz \dskm1 ) 
           \right] \\
    & \qquad + \hat \nu_k (D^k)^{-1} (x^0 - x^*) 
                    + \hat \nu_k {D^k} ( z^0 - z^*) \Big\|^2 .
  \end{split}
\end{equation*}
Using this relation we get 
\begin{equation} 
\label{t:inf:dxbnd} %
  \begin{split}
    \| (D^k)^{-1} \hat x \| 
    & \le \left\| (X^k Z^k)^{-1/2} 
          \left[ Z^k ( \sszk \dxk + \ssz \dxkm1) 
               + X^k ( \sszk \dsk + \ssz \dskm1 ) 
          \right] \right\| \\
    & \quad + \hat \nu_k \| {D^k}^{-1} (x^0 - x^*) \| 
            + \hat \nu_k \| {D^k} ( z^0 - z^*) \|
  \end{split}
\end{equation}
and observe that the same bound holds for $\| D^k \hat z \|$. 
In the next two lemmas, we compute the bounds for the terms 
in the right-hand side of~\eqref{t:inf:dxbnd}.

\begin{lemma} 
\label{l:inf:rdxbnd} %
  Let $k + 1$ be a quasi-Newton iteration of Algorithm~\ref{alg:pdqnipm}. 
  Let $\omega$ and $C_3$ be the constants (independent of $k$ and $n$) 
  defined in equation~\eqref{t:inf:dxdzbnd} and Lemma~\ref{l:inf:gamma1bnd}, 
  respectively. Then
  \begin{equation*}
    \begin{split}
      \Big\| (X^k Z^k)^{-1/2} &
         \left[ Z^k ( \sszk \dxk + \ssz \dxkm1) 
              + X^k ( \sszk \dsk + \ssz \dskm1 ) \right] 
      \Big\| \le \\
      & \le \gamma^{-1/2} 
        \left[ (\sigma_{max} + \gamma^{-1})(\sszk + \ssz) 
             + (\sszk + C_3) \ssz \sszk \omega^2 \right] n^{5/2} \mu_k^{1/2}.
    \end{split}
  \end{equation*}
\end{lemma}

\begin{proof}
  First, we observe that the diagonal matrices $\dXk$ and $\dSk$ were
  computed in the Newton iteration and use~\eqref{t:inf:dxdzbnd} to
  obtain
  \[
    \| \dXk \dSk e \| \le \| (D^k)^{-1} \dXk \| \| D^k \dSk e \| \le
    \| (D^k)^{-1} \dXk \| \omega n \mu_k^{1/2}.
  \]
  Since both $(D^k)^{-1}$ and $\dXk$ are diagonal matrices, using the
  property of the induced 2-norm of matrices, we get 
  \[
    \| (D^k)^{-1} \dXk \| = \underset{i = 1, \dots, n}{\max}
    \dfrac{|[\dxk]_i|}{[D^k]_{ii}} = \| (D^k)^{-1} \dXk e \|_\infty
    \le \| (D^k)^{-1} \dXk e \| \le \omega n \mu_k^{1/2}.
  \]
  Hence,
  \begin{equation} \label{l:inf:rdxbnd:e1} %
    \| \dXk \dSk e \| \le \omega^2 n^2 \mu_k.
  \end{equation}  
  Now, we use equation~\eqref{l:combined:e1} (which does not depend 
  on the feasibility of the iterate or the type of the neighborhood) 
  to expand the desired expression in the statement of this Lemma. 
  Additionally, we use equation~\eqref{t:inf:sufdec} to bound 
  $\mu_{k + 1}$ and $X^k Z^k e$ as well as Lemma~\ref{l:inf:gamma1bnd}
  and equation~\eqref{l:inf:rdxbnd:e1} to derive:
  \begin{equation*}
    \begin{split}
      \left\| (X^k \right. & \left.  Z^k)^{-1/2} 
          \left[ Z^k ( \sszk \dxk + \ssz \dxkm1) 
               + X^k ( \sszk \dsk + \ssz \dskm1 ) \right] \right\| = \\
      & = \left( \sum_{i = 1}^n \dfrac{\left( (1 - \ssz) \sszk \sigma_k \mu_k 
            - (\ssz + \sszk (1 - \ssz)) x_i^k z_i^k + \ssz \sigma_{k + 1} \mu_{k + 1} 
            - (1 +\gamma_1) \ssz \sszk^2 [\dxk]_i [\dsk]_i \right)^2}{x_i^k z_i^k} 
          \right)^{1/2} \\
      & \le (\gamma \mu_k)^{-1/2} \left\| (1 - \ssz) \sszk \sigma_k \mu_k e 
            - (\ssz + \sszk (1 - \ssz)) X^k Z^k e + \ssz \sigma_{k + 1} \mu_{k + 1} e 
            - (1 +\gamma_1) \ssz \sszk^2 \dXk \dSk e \right\| \\
      & \le (\gamma \mu_k)^{-1/2} \left[ (1 - \ssz) \sszk \sigma_k \sqrt{n} \mu_k 
            + (\ssz + \sszk (1 - \ssz)) \gamma^{-1} \sqrt{n} \mu_k 
            + \ssz \sigma_{k + 1} \sqrt{n} \mu_k 
            + |1 +\gamma_1| \ssz \sszk^2 \omega^2 n^2 \mu_k \right] \\
      & \le \gamma^{-1/2} \left[ 
            \left( (\sigma_k + \gamma^{-1})(1 - \ssz) \sszk 
                 + (\sigma_{k + 1} + \gamma^{-1}) \ssz \right) \sqrt{n} 
           + \left(1 + C_3 \frac{\sqrt{n}}{\sszk} \right) 
                 \ssz \sszk^2 \omega^2 n^2 \right] \mu_k^{1/2} \\
      & \le \gamma^{-1/2} \left[ (\sigma_{max} +  \gamma^{-1})(\sszk + \ssz) \sqrt{n} 
           + (\sszk + C_3) \ssz \sszk \omega^2 n^{5/2} \right] \mu_k^{1/2} \\
      & \le \gamma^{-1/2} \left[ (\sigma_{max} + \gamma^{-1})(\sszk + \ssz) 
           + (\sszk + C_3) \ssz \sszk \omega^2 \right] n^{5/2} \mu_k^{1/2}. \\
    \end{split}
  \end{equation*}
\end{proof}

\begin{lemma} \label{l:inf:dxbnd} %
  Let $k + 1$ be a quasi-Newton iteration of
  Algorithm~\ref{alg:pdqnipm}. Let $\omega$ and $C_3$ be the constants
  (independent of $k$ and $n$) defined before. Then,
  \[
    \begin{aligned}
      \| (D^k)^{-1} (\sszk \dxk + \ssz \dxkm1) \| 
        & \le \left[
           (\sigma_{max} + \gamma^{-1} + 8 \beta) (\sszk + \ssz) +
           (\sszk + C_3) \omega^2 \sszk \ssz 
              \right] \gamma^{-1/2} n^{5/2} \mu_k^{1/2} \\
      \| D^k (\sszk \dsk + \ssz \dskm1) \| 
        & \le \left[ (\sigma_{max} + \gamma^{-1} + 8 \beta) (\sszk + \ssz) 
                   + (\sszk + C_3) \omega^2 \sszk \ssz 
              \right] \gamma^{-1/2} n^{5/2} \mu_k^{1/2}.
    \end{aligned}
  \]
\end{lemma}

\begin{proof}
  We will only consider $\| (D^k)^{-1} (\sszk \dxk + \ssz \dxkm1) \|$,
  since getting a bound for $\| D^k (\sszk \dsk + \ssz \dskm1) \|$ 
  follows the same arguments. We use the definition of $\hat x$ (see \eqref{t:inf:scalarp}), 
  add and subtract $\hat \nu_k (D^k)^{-1} (x^0 - x^*)$ inside the norm, 
  and then use the triangle inequality and~\eqref{t:inf:dxbnd} to obtain
  \begin{equation} \label{l:inf:dxbnd:e1} %
    \begin{split}
      \| (D^k)^{-1} (\sszk \dxk + \ssz \dxkm1) \| & = \| (D^k)^{-1}
      \hat x - \hat \nu_k (D^k)^{-1} (x^0 - x^*) \| \le \| (D^k)^{-1}
      \hat x \| + \hat \nu_k \| (D^k)^{-1} (x^0 - x^*) \| \\
      & \le \left\| (X^k Z^k)^{-1/2} \left[ Z^k ( \sszk \dxk + \ssz
          \dxkm1) + X^k ( \sszk \dsk + \ssz \dskm1 ) \right] \right\| \\
      & \quad + 2 \hat \nu_k \| (D^k)^{-1} (x^0 - x^*) \| + 2 \hat
      \nu_k \| {D^k} ( z^0 - z^*) \|,
    \end{split}
  \end{equation}
  where another term $\hat \nu_k \| {D^k} ( z^0 - z^*) \|$ was added
  in the last inequality. We already have bounds for the first term
  in the right-hand side of~\eqref{l:inf:dxbnd:e1}, by
  Lemma~\ref{l:inf:rdxbnd}. From~\cite[Lemma~6.6]{Wright1997} we know
  that
  \begin{equation*}
    \begin{split}
      \| (D^k)^{-1} (x^0 - x^*) \| & \le \| (x^0 - x^*) \| \|
      (D^k)^{-1} \| \le \xi \| (D^k)^{-1} \| = \xi\ \underset{i = 1,
        \dots, n}{\max}\ \dfrac{1}{[D^k]_{ii}} = \xi \| (D^k)^{-1} e
      \|_\infty \\
      & \le \xi \| (D^k)^{-1} e \| = \xi \| (X^k Z^k)^{-1/2} Z^k e \|
      \le \xi \| (X^k Z^k)^{-1/2} \| \| z^k \|_1
    \end{split}
  \end{equation*}
  and, similarly,
  $\| {D^k} ( z^0 - z^*) \| \le \xi \| (X^k Z^k)^{-1/2} \| \| x^k
  \|_1$. We recall that
  $\hat \nu_k = (1 - (1 - \sszk)(1 - \ssz)) \nu_k$ and apply all 
  these inequalities together with~\eqref{t:inf:xz1bnd} to obtain
  \begin{equation*}
    \begin{split}
      2 \hat \nu_k \| (D^k)^{-1} (x^0 - x^*) \| + 2 \hat \nu_k \|
      {D^k} ( z^0 - z^*) \| & \le 2 \xi (1 - (1 - \sszk)(1 - \ssz))
      \|(X^k Z^k)^{-1/2} \| \nu_k \| (x^k, z^k) \|_1 \\
      & \le 2 \xi (1 - (1 - \sszk)(1 - \ssz)) (\gamma \mu_k)^{-1/2}
      \frac{4 \beta}{\xi} n \mu_k \\
      & = (1 - (1 - \sszk)(1 - \ssz)) \frac{8 \beta}{\gamma^{1/2}} n
      \mu_k^{1/2},
    \end{split}
  \end{equation*}
  which provides bounds for the last two terms in~\eqref{l:inf:dxbnd:e1}. 
  Therefore, using Lemma~\ref{l:inf:rdxbnd} and the above inequality we write 
  \begin{equation*}
    \begin{split}
      \| {D^k}^{-1} (\sszk \dxk + \ssz \dxkm1) \| 
      & \le \left[ \frac{(\sigma_{max} +\gamma^{-1}) (\sszk + \ssz) 
             + (\sszk + C_3) \ssz \sszk \omega^2 }{\gamma^{1/2}} n^{5/2} 
             + \frac{8 \beta (1 - (1 - \sszk)(1 - \ssz)) }{\gamma^{1/2}} n 
            \right] \mu_k^{1/2} \\
      & \le \frac{ (\sigma_{max} + \gamma^{-1}) 
             (\sszk + \ssz) + (\sszk + C_3) \ssz \sszk \omega^2 
             + 8 \beta (1 - (1 - \sszk)(1 - \ssz))}{\gamma^{1/2}} n^{5/2} \mu_k^{1/2} \\
      & \le \left[ (\sigma_{max} + \gamma^{-1} 
             + 8 \beta) (\sszk + \ssz) + (\sszk + C_3) \omega^2 \sszk \ssz 
            \right] \gamma^{-1/2} n^{5/2} \mu_k^{1/2}
    \end{split}
  \end{equation*}
  and the lemma is proved.
\end{proof}

If we restrict the choices of $\sszk$ and $\ssz$, then the bounds 
obtained in Lemma~\ref{l:inf:dxbnd} can be significantly simplified
as shown in the corollary below.

\begin{corollary} \label{c:inf:dxbnd} %
  If
  $\sszk \le (\sigma_{max} + \gamma^{-1} + 8 \beta) 
   ((1 + C_3) \omega^2)^{-1}$ 
  and $\ssz \le C_3^{-1} \sszk$, then there exists a constant $C_6$, 
  independent of $n$, such that
  \[
    \| (D^k)^{-1} (\sszk \dxk + \ssz \dxkm1) \| 
       \le C_6 \sszk n^{5/2} \mu_k^{1/2} 
       \quad \text{and} \quad 
    \| D^k (\sszk \dsk + \ssz \dskm1) \| 
       \le C_6 \sszk n^{5/2} \mu_k^{1/2}.
  \]
\end{corollary}

\begin{proof}
  Using the bounds on $\ssz$ and $\sszk$ assumed in the lemma we obtain 
  $(\sszk + C_3) \omega^2 \sszk \le (\sigma_{max} + \gamma^{-1} + 8 \beta)$, 
  hence, by $\ssz \le C_3^{-1} \sszk$,
  \begin{equation*}
    \begin{split}
      (\sigma_{max} + \gamma^{-1} + 8 \beta) (\sszk + \ssz) 
        + (\sszk + C_3) \omega^2 \sszk \ssz 
      & \le (\sigma_{max} + \gamma^{-1} + 8 \beta) (\sszk + 2 \ssz) \\
      & \le (\sigma_{max} + \gamma^{-1} + 8 \beta) (1 + 2 C_3^{-1}) \sszk,
    \end{split}
  \end{equation*}
  and the conclusion follows by defining
  $C_6 = (\sigma_{max} + \gamma^{-1} + 8 \beta) 
                          (1 + 2 C_3^{-1}) \gamma^{-1/2}$.
\end{proof}

We are now ready to prove the polynomial worst-case iteration
complexity of Algorithm~\ref{alg:pdqnipm} in the infeasible case.
Lemma~\ref{l:inf:symmetric} dealt with the Newton step of the method.
Theorems~\ref{t:inf:complex} and~\ref{t:inf:complex2} provide the
results for the quasi-Newton step.

\begin{theorem} \label{t:inf:complex} %
  Suppose that $k + 1$ is a quasi-Newton step of
  Algorithm~\ref{alg:pdqnipm} and all the hypotheses of this section
  hold. If the following condition holds 
  \begin{equation} \label{t:inf:complex:p} %
    \begin{aligned}
      \alphadec + \sigma_{max} & \le 1 - \sigma_{min}
    \end{aligned}
  \end{equation}
  then, there exists a constant $C_5$, independent of $n$, such that, if 
  \begin{equation}
    \label{t:inf:complex:lims}
    \sszk \in \left(0, \min \left\{ 1, \left((1 + C_3) \omega^2 \right)^{-1}, 
                                       \frac{C_5}{n^5} \right\} \right] 
       \quad \text{and} \quad
    \ssz \in \left[(C_3 C_5)^{-1} n^5 \sszk^2, \ C_3^{-1} \sszk \right],
  \end{equation}
  then the iterate after the quasi-Newton step satisfies 
  \begin{align}
    (x^{k + 1} + \ssz \dxkm1)^T(z^{k + 1} + \ssz \dskm1) 
       & \ge (1 - \ssz) {x^{k + 1}}^T z^{k + 1} \label{t:inf:complex:e1} \\
    (x^{k + 1} + \ssz \dxkm1)^T(z^{k + 1} + \ssz \dskm1) 
       & \le (1 - \alphadec \ssz) {x^{k + 1}}^T z^{k + 1} \label{t:inf:complex:e2} \\
    [x^k + \sszk \dxk]_i [z^k + \sszk \dsk]_i 
       & \ge \gamma \mu(\ssz) \label{t:inf:complex:e3} \\
    [x^k + \sszk \dxk]_i [z^k + \sszk \dsk]_i 
       & \le \frac{1}{\gamma} \mu(\ssz) \label{t:inf:complex:e4}.
  \end{align}
\end{theorem}

\begin{proof}
  By Lemma~\ref{l:inf:gamma1bnd} we know that $C_3 \ge 1$. 
  Hence, the given intervals for $\ssz$
  and $\sszk$ are well defined. We begin the proof with
  inequality~\eqref{t:inf:complex:e1}.  Using~\eqref{t:inf:dxdzbnd},
  we get
  \begin{equation}
    \label{t:inf:c:e3}
    \sszk^2 {\dxk}^T \dsk 
       \le \sszk^2 \| {D^k}^{-1} \dxk \| \| D^k \dsk \| 
       \le \sszk^2 \omega^2 n^2 \mu_k.
  \end{equation}
  By Corollary~\ref{c:inf:dxbnd}, we obtain 
  \begin{equation}
    \label{t:inf:c:e4}
    \begin{split}
      (\sszk \dxk + \ssz \dxkm1)^T (\sszk \dsk + \ssz \dskm1) 
      & \le \| (D^k)^{-1} (\sszk \dxk + \ssz \dxkm1) \| \| D^k (\sszk \dsk + \ssz \dskm1) \| \\
      & \le C_6^2 n^5 \mu_k \sszk^2.
    \end{split}
  \end{equation}
  Finally, by using inequality~\eqref{l:inf:sym:e0} from
  Lemma~\ref{l:inf:symmetric}, we get 
  \begin{equation}
    \label{t:inf:c:e5}
    \ssz \sigma_{k + 1} {x^{k + 1}}^T z^{k + 1} \ge \ssz (1 - \sszk)
    \sigma_{k + 1} {x^k}^T z^k \ge \sigma_{min} \ssz
    (1 - \sszk) n \mu_k.
  \end{equation}
  Starting from equation~\eqref{CP-afterQN}, using~\eqref{t:inf:c:e3}--\eqref{t:inf:c:e5},  
  and then applying the bounds for $|\gamma_1|$ from Lemma~\ref{l:inf:gamma1bnd} 
  together with the bound $\ssz \le C_3^{-1} \sszk$ (assumed in~\eqref{t:inf:complex:lims})
  to deliver $1 + \ssz |\gamma_1| \leq 1 + C_3^{-1} \sszk C_3 \frac{\sqrt{n}}{\sszk} = 1 + \sqrt{n}$, 
  we get 
  \begin{equation*}
    \begin{split}
      (x^{k + 1} & + \ssz \dxkm1)^T (z^{k + 1} + \ssz \dskm1) - (1 -
      \ssz) {x^{k + 1}}^T z^{k + 1} \ge \\
      & \ge \sigma_{min} \ssz (1 - \sszk) n \mu_k - C_6^2 n^5 \mu_k
      \sszk^2 - (1 + \gamma_1 \ssz) \sszk^2 {\dxk}^T \dsk \\
      & \ge \sigma_{min} \ssz (1 - \sszk) n \mu_k - (1 + \ssz
      |\gamma_1|) \sszk^2 \lvert {\dxk}^T \dsk \rvert - C_6^2 n^5 \mu_k \sszk^2 \\
      & \ge \left( \sigma_{min} \ssz - \sigma_{min} \ssz \sszk -
        (1 + n^{1/2}) \sszk^2 \omega^2 n - C_6^2 n^4 \sszk^2 \right) n \mu_k \\
      & \ge \left[ \sigma_{min} \ssz - \left( C_3^{-1} \sigma_{min} +
          2 \omega^2 + C_6^2 \right) \sszk^2 n^4 \right] n \mu_k.
    \end{split}
  \end{equation*}
  By defining
  \[
    \kappa = 2 \omega^2 + C_6^2 
       \quad \text{and} \quad 
    C_4 = C_3^{-1} \sigma_{min} \left( C_3^{-1} \sigma_{min} + \kappa \right)^{-1},
  \]
  the right-hand side of the previous inequality is non-negative 
  if $\ssz \ge (C_4 C_3)^{-1} n^4 \sszk^2$ and $\sszk \le C_4 / n^4$.
  The lower bound for $\ssz$ is necessary to avoid too small step 
  in the quasi-Newton direction. It is obtained by writing $\ssz$ 
  as a function of $\sszk$, in order to guarantee its non-negativity. 
  The upper bound for $\sszk$ is obtained by requesting that $\sszk$ 
  has to be chosen such that $(C_4 C_3)^{-1} n^4 \sszk^2 \le C_3^{-1} \sszk$
  holds.
  %
  %

  For inequality~\eqref{t:inf:complex:e2}, we start by subtracting
  $(1 - \alphadec \ssz) {x^{k + 1}}^T z^{k + 1}$ from both sides
  of equation~\eqref{CP-afterQN}
  \begin{equation*}
    \begin{split}
      (x^{k + 1} & + \ssz \dxkm1)^T (z^{k + 1} + \ssz \dskm1) 
         - (1 - \alphadec \ssz) {x^{k + 1}}^T z^{k + 1} = \\
      & = - (1 - \alphadec - \sigma_{k + 1}) \ssz {x^{k + 1}}^T z^{k + 1} 
          + (\sszk \dxk + \ssz \dxkm1)^T (\sszk \dsk + \ssz \dskm1) 
          - (1 + \ssz \gamma_1) \sszk^2 {\dxk}^T \dsk.
    \end{split}
  \end{equation*}
  We observe that the first term is negative by the assumption 
  \eqref{t:inf:complex:p} of the Theorem, since $\sigma_{k + 1} \le \sigma_{max}$. 
  Next, from the left inequality in~\eqref{l:inf:sym:e0}, we have 
  $(1 - \bar \alpha_k ) n \mu_k \leq n \mu_{k+1}$ and then by using 
  the condition $\alphadec + \sigma_{max} \le 1 - \sigma_{min}$, 
  \eqref{t:inf:c:e3} and~\eqref{t:inf:c:e4}, we conclude that
  \begin{equation*}
    \begin{split}
      (x^{k + 1} 
      & + \ssz \dxkm1)^T (z^{k + 1} + \ssz \dskm1) 
        - (1 - \alphadec \ssz) {x^{k + 1}}^T z^{k + 1} \le \\
      & \le - \sigma_{min} \ssz (1 - \sszk) n \mu_k 
        + (1 + \ssz |\gamma_1|) \sszk^2 |{\dxk}^T \dsk| 
        + |(\sszk \dxk + \ssz \dxkm1)^T (\sszk \dsk + \ssz \dskm1)| \\
      & \le - \sigma_{min} \ssz (1 - \sszk) n \mu_k 
        + (1 + n^{1/2}) \sszk^2 \omega^2 n^2 \mu_k + C_6^2 \sszk^2 n^5 \mu_k \\
      & \le \left[ - \sigma_{min} \ssz 
        + \left(\sigma_{min} C_3^{-1} + \kappa \right) \sszk^2 n^4 
            \right] n \mu_k, 
    \end{split}
  \end{equation*}
  where in the last inequality we used the bound 
  $\ssz \leq C_3^{-1} \sszk$ to deliver 
  $\sigma_{min} \ssz \sszk n \mu_k \leq \sigma_{min} C_3^{-1} \sszk^2 n \mu_k$. 
  The right-hand side of the previous inequality will be non-positive if, 
  as before, $\ssz \ge (C_4 C_3)^{-1} n^4 \sszk^2$ and $\sszk \le C_4 / n^4$. 

  To prove inequalities~\eqref{t:inf:complex:e3} and~\eqref{t:inf:complex:e4}, 
  we first look at~\eqref{singleProd} which is a component-wise version 
  of~\eqref{NewComplProds}. We also need to derive component-wise versions 
  of~\eqref{t:inf:c:e4} and~\eqref{t:inf:c:e3}. 
  For~\eqref{t:inf:c:e4} we have 
  \begin{equation} \label{t:inf:c:e6} %
    \begin{split}
      [\sszk \dxk + \ssz \dxkm1]_i [\sszk \dsk + \ssz \dskm1]_i 
      & \le \lvert [(D^k)^{-1}]_{ii} [\sszk \dxk + \ssz \dxkm1]_i \rvert \
            \lvert [D^k]_{ii} [\sszk \dsk + \ssz \dskm1]_i \rvert \\
      & \le \| (D^k)^{-1} (\sszk \dxk + \ssz \dxkm1) \| \| D^k (\sszk \dsk + \ssz \dskm1) \| \\
      & \le C_6^2 n^5 \mu_k \sszk^2,
    \end{split}
  \end{equation}
  and, by a similar approach, equation~\eqref{t:inf:c:e3},
  Lemma~\ref{l:inf:gamma1bnd} and the assumption of the theorem 
  $\ssz \le C_3^{-1} \sszk$ we get 
  \begin{equation} \label{t:inf:c:e7} %
    \begin{split}
      (1 + \ssz \gamma_1) \sszk^2 [\dxk]_i [\dsk]_i & \le (1 + \ssz
      \lvert \gamma_1 \rvert ) \sszk^2 \lvert [\dxk]_i [\dsk]_i \rvert \\
      & \le (1 + n^{1/2}) \sszk^2 \| (D^k)^{-1} \dXk \| \| D^k \dSk
      \| \\
      & \le 2 \sszk^2 \omega^2 n^{5/2} \mu_k.
    \end{split}
  \end{equation}
  Now, taking~\eqref{singleProd} and applying~\eqref{l:inf:sym:e1}
  and~\eqref{t:inf:c:e6}--\eqref{t:inf:c:e7} we get 
  \begin{equation*}
    \begin{split}
      [x^{k + 1} + & \ssz \dxkm1]_i [z^{k + 1} + \ssz \dskm1]_i \ge \\
      & \ge (1 - \ssz) \gamma \mu_{k + 1} - \lvert [\sszk \dxk + \ssz
      \dxkm1]_i [\sszk \dxk + \ssz \dskm1]_i \rvert - (1 + \ssz \lvert
      \gamma_1 \rvert) \sszk^2 \lvert [\dxk]_i [\dsk]_i \rvert + \ssz
      \sigma_{k + 1} \mu_{k + 1} \\
      & \ge (1 - \ssz) \gamma \mu_{k + 1} + \ssz \sigma_{k + 1} \mu_{k
        + 1} - C_6^2 n^5 \mu_k \sszk^2 - 2 \sszk^2 \omega^2 n^{5/2}
      \mu_k \\
      & \ge (1 - \ssz) \gamma \mu_{k + 1} + \ssz \sigma_{k + 1} \mu_{k
        + 1} - \kappa n^5 \sszk^2 \mu_k
    \end{split}
  \end{equation*}
  and from~\eqref{CP-afterQN}, using~\eqref{t:inf:c:e3}--\eqref{t:inf:c:e4}
  we get 
  \begin{equation*}
    \begin{split}
      \gamma \mu(\ssz) & = \gamma \frac{(x^{k + 1} + \ssz \dxkm1)^T
        (z^{k + 1} + \ssz \dskm1)}{n} \le \\
      & \le \gamma \left[ (1 - \ssz (1 - \sigma_{k + 1}) ) \mu_{k + 1}
        + \frac{1 + \ssz \lvert \gamma_1 \rvert}{n} \sszk^2 \lvert
        {\dxk}^T \dsk \rvert + \frac{\lvert (\sszk \dxk + \ssz
          \dxkm1)^T (\sszk \dsk + \ssz \dskm1) \rvert }{n}
      \right] \\
      & \le (1 - \ssz) \gamma \mu_{k + 1} + \ssz \sigma_{k + 1} \gamma
      \mu_{k + 1} + (1 + n^{1/2}) \gamma \omega^2 n \sszk^2 \mu_k +
      C_6^2 \gamma n^4  \sszk^2 \mu_k \\
      & \le (1 - \ssz) \gamma \mu_{k + 1} + \ssz \sigma_{k + 1} \gamma
      \mu_{k + 1} + \kappa \gamma n^4 \sszk^2 \mu_k.
    \end{split}
  \end{equation*}
  By combining the above two inequalities, then
  using~\eqref{t:inf:c:e5} and~\eqref{l:inf:sym:e0} and
  $\ssz \le C_3^{-1} \sszk$, we obtain
  \begin{equation*}
    \begin{split}
      [x^{k + 1} + \ssz \dxkm1]_i [z^{k + 1} + \ssz \dskm1]_i - \gamma
      \mu(\ssz) & \ge (1 - \gamma) \ssz \sigma_{k + 1} \mu_{k + 1} -
      (1 + \gamma) \kappa n^5 \sszk^2 \mu_k \\
      & \ge (1 - \gamma)
      \sigma_{min} \ssz (1 - \sszk) \mu_k - (1 +
      \gamma) \kappa n^5 \sszk^2 \mu_k \\
      & \ge \left[ \sigma_{min} \ssz - \left( C_3^{-1} \sigma_{min} +
          \left( \frac{1 + \gamma}{1 - \gamma} \right) \kappa \right)
        n^5 \sszk^2 \right] (1 - \gamma) \mu_k.
    \end{split}
  \end{equation*}
  The right-hand side of this inequality is non-negative, and
  hence~\eqref{t:inf:complex:e3} holds, if
  $\ssz \ge (C_3 C_5)^{-1} n^5 \sszk^2$ and $\sszk \le C_5 / n^5$,
  with
  \[
    C_5 = C_3^{-1} \sigma_{min} \left(C_3^{-1} \sigma_{min} + \left(
        \frac{1 + \gamma}{1 - \gamma} \right) \kappa \right)^{-1}.
  \]
  Finally, in order to show~\eqref{t:inf:complex:e4}, we use the same
  ideas to obtain the following two inequalities
  \[
    [x^{k + 1} + \ssz \dxkm1]_i [z^{k + 1} + \ssz \dskm1]_i 
    \le \frac{1 - \ssz}{\gamma} \mu_{k + 1} 
        + \ssz \sigma_{k + 1} \mu_{k + 1} + \kappa n^5 \sszk^2 \mu_k
  \]
  and
  \[
    \frac{1}{\gamma} \mu(\ssz) \ge \frac{1 - \ssz}{\gamma} \mu_{k + 1}
    + \frac{\ssz \sigma_{k + 1}}{\gamma} \mu_{k + 1} 
    - \frac{\kappa n^4 \sszk^2}{\gamma} \mu_k.
  \]
  By combining these two inequalities and 
  using~\eqref{l:inf:sym:e0}
  and $\ssz \le C_3^{-1} \sszk$ once more, we have that
  \begin{equation*}
    \begin{split}
      [x^{k + 1} + \ssz \dxkm1]_i [z^{k + 1} + \ssz \dskm1]_i - \frac{1}{\gamma} \mu(\ssz) 
      & \le \left(1 - \frac{1}{\gamma} \right) \sigma_{k + 1} \ssz \mu_{k + 1} 
            + \left( 1 + \frac{1}{\gamma} \right) \kappa n^5 \sszk^2 \mu_k \\
      & = - \left( \frac{1 - \gamma}{\gamma} \right) \sigma_{k + 1} \ssz \mu_{k + 1} 
          + \left( \frac{1 + \gamma}{\gamma} \right) \kappa n^5 \sszk^2 \mu_k \\
      & \le - \left( \frac{1 - \gamma}{\gamma} \right) \sigma_{min} \ssz (1 - \sszk) \mu_k 
          + \left( \frac{1 + \gamma}{\gamma} \right) \kappa n^5 \sszk^2 \mu_k \\
      & \le \left[ - \sigma_{min} \ssz 
          + \left( C_3^{-1} \sigma_{min} 
              + \left( \frac{1 + \gamma}{1 - \gamma} \right) \kappa \right) n^5 \sszk^2 
            \right] \left( \frac{1 - \gamma}{\gamma} \right) \mu_k.
    \end{split}
  \end{equation*}
  Again, the right-hand side of this inequality is non-positive if
  $\ssz \ge (C_3 C_5)^{-1} n^5 \sszk^2$ and $\sszk \le C_5 / n^5$,
  with $C_5$ defined before.
  Since $C_5 \leq C_4$, we observe that
  $(C_3 C_5)^{-1} n^5 \ge (C_3 C_4)^{-1} n^5$,
  and this explains the lower bound on $\ssz$. 
  We also observe that in~\eqref{t:inf:complex:lims}, 
  $\sszk$ is bounded from above by $C_5/n^5$ 
  hence $(C_3 C_5)^{-1} n^5 \sszk^2 \leq C_3^{-1} \sszk$, 
  which guarantees that the interval for $\ssz$ is not empty. 
  Similar observations can be made to justify the upper bound 
  of $\sszk$ in~\eqref{t:inf:complex:lims} and this concludes 
  the proof. 
\end{proof}
  To further show that $\xysk{k + 2}$ belongs to the
  $\nN_s(\gamma, \beta)$ neighborhood, and therefore
  satisfies~\eqref{t:inf:sufdec}, we need to ensure that $\gamma$ is not
  too close to 0.
\begin{theorem} \label{t:inf:complex2} %
  Let the hypotheses of Theorem~\ref{t:inf:complex} hold. If, in
  addition, we request that
  \begin{equation} \label{t:inf:c2:bnds} %
    \gamma \ge 2 \left( - 8 \beta + \sqrt{(8 \beta + 2)^2 + \frac{4}{3
          \sigma_{min}}} \right)^{-1}
  \end{equation}
  then $\xysk{k + 2}$ satisfies~\eqref{t:inf:sufdec}.
\end{theorem}

\begin{proof}
  Using inequality~\eqref{t:inf:complex:e1} in
  Theorem~\ref{t:inf:complex}
  \[
    \frac{\| (r^{k + 2}_b, r^{k + 2}_c) \|}{\mu_{k + 2}} \le \frac{(1
      - \ssz) \| (r^{k + 1}_b, r^{k + 1}_c) \|}{(1 - \ssz) \mu_{k +
        1}} \le \frac{\| (r^0_b, r^0_c) \|}{\mu_0} \beta
  \]
  and, by~\eqref{t:inf:complex:e3} and~\eqref{t:inf:complex:e4}
  \[
    0 <   \gamma \mu_{k + 2} 
    \le x^{k + 2}_i z^{k + 2}_i 
    \le \frac{1}{\gamma} \mu_{k + 2}. 
  \]
  To show that $x^{k + 2} > 0$ and $z^{k + 2} > 0$, we suppose by
  contradiction that $x^{k + 2}_i \le 0$ or $z^{k + 2}_i \le 0$ occurs
  for some $i$ and follow the same arguments as those used in
  Theorems~\ref{t:n2conv} and~\ref{t:nsconv} to conclude,
  using inequality~\eqref{t:inf:c:e6}, that $\gamma < C_6^2 n^5 \sszk^2$ 
  should hold.  

  By using the fact that $C_3 \ge 1$, defined in
  Lemma~\ref{l:inf:gamma1bnd}, and $\kappa \ge 1$ and $C_6 \ge 1$,
  defined in Theorem~\ref{t:inf:complex}, and basic manipulation, we
  conclude that $C_6 \le 3 (1 + \gamma^{-1} + 8 \beta) \gamma^{-1/2}$ and
  $C_5 \le \sigma_{min} (1 - \gamma)$, where $C_5$ was also defined in
  Theorem~\ref{t:inf:complex}. Then, using \eqref{t:inf:complex:lims}
  for $\sszk$ we obtain
  \begin{equation} \label{t:inf:c2:gamma}
    \gamma < C_6^2 n^5 \sszk^2 \le \frac{(C_5 C_6)^2}{n^5} 
      \le \frac{9 (1 + \gamma^{-1} + 8 \beta)^2 (1 - \gamma)^2}{\gamma} \sigma_{min}^2
  \end{equation} 
  which implies that
  $ \gamma < 2 \left(- 8 \beta + \sqrt{(8\beta + 2)^2 
             + \frac{4}{3 \sigma_{min}}} \right)^{-1}$ 
  and contradicts~\eqref{t:inf:c2:bnds}.
  The above inequality has been obtained by
  rearranging~\eqref{t:inf:c2:gamma} namely, dropping the squares and
  solving the resulting (quadratic) inequality with an unknown
  $1/\gamma$.
  Therefore, we conclude that $\xysk{k + 2} \in \nN_s(\gamma, \beta)$.
  Earlier proved inequality \eqref{t:inf:complex:e2} guarantees 
  that iteration $k + 1$ satisfies~\eqref{t:inf:sufdec}.

\end{proof}

%
%

For sufficiently large $n$, we clearly have 
\[
  \min \left\{ 1, \left((1 + C_3) \omega^2 \right)^{-1}, \frac{C_5}{n^5} \right\} 
     = \frac{C_5}{n^5} 
     < \frac{\bar {\bar \delta}}{n^2},
\]
where $\bar {\bar \delta}$ comes from Lemma~\ref{l:inf:symmetric} 
and this guarantees that the hypotheses of this section hold. 
By~\eqref{t:inf:complex:lims} we have that both $\sszk$ and
$\ssz$ are greater than or equal to $\frac{C_3^{-1}C_5}{n^5}$ 
hence
\[
  \mu_{k + 1} 
  \le (1 - \alphadec \sszk) \mu_k 
  \le \left( 1 - \frac{\alphadec C_3^{-1}C_5}{n^5} \right) \mu_k, 
\]
for all $k = 0, 1, \dots$, which yields $O(n^5)$ worst-case iteration 
complexity for the infeasible case of Algorithm~\ref{alg:pdqnipm}.

Unlike in the feasible case, because of the close relation of the
symmetric neighborhood and the $\nN_{-\infty}(\gamma,\beta)$
neighborhood, we expect that a similar complexity result to that of
Theorem~\ref{t:inf:complex} should hold for the algorithm which
operates in $\nN_{-\infty}(\gamma,\beta)$. However, to save space we
do not derive it here.

%
%
%

\section{Conclusions and final observations} \label{remarks}

This work provided theoretical tools to analyze the worst-case
iteration complexity of quasi-Newton primal-dual interior point
algorithms. A simplified algorithm was considered, which consisted of
alternating Newton and quasi-Newton steps. The quasi-Newton approach
was based on the Broyden ``bad'' low-rank update of the inverse of the
unreduced system. Feasible and infeasible algorithms and well
established neighborhoods of the central path have been considered.

The results showed that in all cases, the degree of the polynomial 
in the worst-case result has increased. This behavior has already 
been observed by~\cite{GondzioSobral2019}, where the number 
of overall iterations increased, but the number of factorizations 
of the true unreduced system actually decreased.

Complexity results might also be obtained in the feasible case 
for convex quadratic programming problems and the $\nN_2(\theta)$
neighborhood, following~\cite{Gondzio2012a}, for example. However,
since Lemma~\ref{l:ortho} does not hold anymore and thus $\gamma_1$ 
is not eliminated from $\mu(\ssz)$, a further study needs to be 
carried out.

Another interesting and complicated question is the case $\ell > 1$,
where we allow more than one quasi-Newton step after a Newton
step. The authors in~\cite{GondzioSobral2019} observed a trade-off
between the number of consecutive quasi-Newton iterations and 
the overall speed of convergence of the algorithm. Numerical 
results showed that $\ell = 5$ is a reasonable choice. 
We expect that similar worst-case complexity bounds should 
be obtained for the case $\ell > 1$, with possibly higher 
degrees of polynomials, but this still remains an open question.

The worst-case complexity results obtained when the $\nN_s$ 
neighborhood was considered in both feasible and infeasible cases 
seem rather pessimistic. The high polynomial degrees originate 
from Lemma~\ref{l:ns:qnstepbnd} for the feasible and 
from Lemma~\ref{l:inf:dxbnd} for the infeasible cases. 
Finding new ways of reducing the degrees of such expressions 
would definitely result in better complexity results, since 
all the other terms are at most of order $n$.

\bibliographystyle{jabbrv_siam}
\bibliography{gs2}

\newpage

\end{document}